\documentclass[reqno,12pt]{amsart}
\usepackage[active]{srcltx}
\usepackage[utf8]{inputenc}
\usepackage{calrsfs}
\usepackage{amssymb}
\usepackage[pdftex]{hyperref}
\usepackage{color}
\usepackage[showonlyrefs]{mathtools}
\usepackage{graphicx}
\usepackage{enumerate}
\usepackage{booktabs}
\usepackage[small]{caption}

\usepackage{amsmath}
\usepackage{amsthm}
\usepackage{amsfonts}
\usepackage{amssymb}
\usepackage{mathrsfs}

\mathtoolsset{showonlyrefs=true}

\numberwithin{equation}{section}
\newtheorem{theorem}{Theorem}
\numberwithin{theorem}{section}
\newtheorem{lem}[theorem]{Lemma}
\newtheorem{cor}[theorem]{Corollary}
\newtheorem{prop}[theorem]{Proposition}
\newtheorem{defi}[theorem]{Definition}

\newtheorem{obs}[theorem]{Remark}

\newcommand{\R}{\ensuremath{\mathbb{R}}}
\newcommand{\N}{\ensuremath{\mathbb{N}}}

\begin{document}
\title{On Hamiltonian systems with critical Sobolev exponents}

\author[A. Guimar\~aes]{Angelo Guimar\~aes}
\address[Angelo Guimar\~aes]{Instituto de Ci\^encias Matem\'aticas e de Computa\c{c}\~ao \newline \indent Universidade de S\~ao Paulo \newline \indent Caixa Postal 668, CEP 13560-970 - S\~ao Carlos - SP - Brazil}
\email{\tt g.angelo@usp.br}

\author[E. M. dos Santos]{Ederson Moreira dos santos}
\address[Ederson Moreira dos Santos]{Instituto de Ci\^encias Matem\'aticas e de Computa\c{c}\~ao \newline \indent Universidade de S\~ao Paulo \newline \indent Caixa Postal 668, CEP 13560-970 - S\~ao Carlos - SP - Brazil}
\email{\tt ederson@icmc.usp.br}

\date{\today}

\keywords{Lane-Emden systems; Critical hyperbola; Critical dimension; Positive solutions.}

\subjclass[2010]{35J47, 35J30, 35B33. }

\thanks{E.M. dos Santos has been partially supported by CNPq grant 309006/2019-8.  A. Guimar\~aes has been supported by CAPES (Finance code 001). We thank Professors J\'essyca L. F. Melo Gurj\~ao and Marcelo Furtado for some enlightening discussions at the beginning of this project.}

\begin{abstract}
In this paper we consider lower order perturbations of the critical Lane-Emden system posed on a bounded smooth domain $\Omega \subset \mathbb{R}^N$, with $N \geq3$, inspired by the classical results of Brezis and Nirenberg \cite{BrezisNirenberg1983}. We solve the problem of finding a positive solution for all dimensions $N \geq 4$. For the critical dimension $N=3$ we show a new phenomenon, not observed for scalar problems. Namely, there are parts on the critical hyperbola where solutions exist for all $1$-homogeneous or subcritical superlinear perturbations and parts where there are no solutions for some of those perturbations.
\end{abstract}

\maketitle

\section{Introduction}

In the memorable paper \cite{BrezisNirenberg1983} from 1983, Brezis and Nirenberg considered the perturbed Lane-Emden equation with critical growth
\begin{equation}\label{eq:BN}
-\Delta u = \lambda u^t + u^{2^*-1}\ \ \text{in} \ \ \Omega, \ \ u>0 \ \ \text{in} \ \ \Omega, \ \ u =0 \ \ \text{on} \ \ \partial \Omega,
\end{equation}
in a bounded smooth domain $\Omega \subset \mathbb{R}^N$, $N \geq 3$, where $2^* = 2N/(N-2)$ is the critical Sobolev exponent for the embedding of $H^1_0(\Omega)$, with $1 \leq t < 2^*-1$. In particular, they discovered a surprising difference between the cases $N\geq 4$ and $N=3$, the latter named as critical dimension. For the particular case with $t=1$, namely for
\begin{equation}\label{eq:BNlinear}
-\Delta u = \lambda u + u^{2^*-1}\ \ \text{in} \ \ \Omega, \ \ u>0 \ \ \text{in} \ \ \Omega, \ \ u =0 \ \ \text{on} \ \ \partial \Omega,
\end{equation}
they proved the existence of a solution for every $0 < \lambda < \lambda_1(\Omega)$, the optimal interval for existence, for $N \geq 4$. In contrast, with $N=3$, they showed the existence of  $0 <\lambda^* < \lambda_1(\Omega)$ such that no solution exists for $0 < \lambda < \lambda^*$; see \cite[Theorem 1.2 and Corollary 1.1]{BrezisNirenberg1983}. Here $\lambda_1= \lambda_1(\Omega)$ stands for the first eigenvalue of $(-\Delta, H^1_0(\Omega))$.

The notion of critical growth for Hamiltonian systems, as independently introduced by Mitidieri \cite{Mitidieri_1993} and van der Vorst \cite{van1992variational}, soon after considered by several authors, including Cl\'ement et al. \cite{Cl_ment_1992} and Peletier-van der Vorst \cite{zbMATH00090502}, is given by the so-called critical hyperbola. In 1998, Hulshof et al. \cite{HMV} analyzed the version of \eqref{eq:BNlinear} in the framework of Hamiltonian systems, namely they considered
\[
\left\{
\begin{array}{lll}
-\Delta u = \lambda v+ |v|^{p-1}v\text{ in } \Omega,\\
-\Delta v = \mu u+ |u|^{q-1}u\text{ in } \Omega,\\
u,v=0 \text{ on } \partial\Omega,
\end{array}
\right.
\]
with $N\geq 4$, for $(p,q)$ on the critical hyperbola
\begin{equation}\label{pqHC}
\frac{1}{p+1} + \frac{1}{q+1}= \frac{N-2}{N}.
\end{equation}

Very interesting results were proved in \cite[Theorem 2]{HMV} and we think that three important problems were left open: 
\begin{itemize}
\item[a)] What does happen in dimension $N=3$\,?
\item[b)] What is the meaning of critical dimension for Hamiltonian elliptic systems\,?
\item[c)] The investigation of the general $1$-homogenous perturbation of the critical Lane-Emden system, namely \eqref{sist} ahead with $rs=1$, which includes $r=s=1$ as a particular case.
\end{itemize}

Item c) deserves some extra comments, since the most accurate 1-homogenous perturbation to Hamiltonian systems, given below in \eqref{sist}, is induced by the hyperbola of points $(r,s)$ such that $rs=1$. Indeed, this hyperbola has been named as the spectral curve  for Hamiltonian systems; see \cite{zbMATH01467776,Leite_2019,Leite_2020} for linear operators and \cite{Santos:2020tv} in the fully nonlinear scenario. In this paper we address these three questions and present some results observed in the framework of Hamiltonian systems which are non-existent for scalar problems. In order to accomplish that, consider the following Hamiltonian system 

\begin{equation}\tag{HS}\label{sist}
\left\{
\begin{array}{lll}
-\Delta u = \lambda |v|^{r-1}v+ |v|^{p-1}v\text{ in } \Omega,\\
-\Delta v = \mu |u|^{s-1}u+ |u|^{q-1}u\text{ in } \Omega,\\
u,v=0 \text{ on } \partial\Omega,
\end{array}
\right.
\end{equation}
in a bounded smooth domain $\Omega \subset \mathbb{R}^N$, $N\geq 3$, $\lambda>0$ and $\mu>0$. Here $(p,q)$ lies on the critical hyperbola, that is $p>0$ and $q>0$ satisfy \eqref{pqHC}, and $(r,s)$ is such that
\begin{equation}\label{rscondition}
0<r <p, \ \ 0<s<q, \ \  rs\geq1.
\end{equation}

Since $\lambda> 0$ and $\mu>0$, the critical growth system \eqref{sist} can be seen as a lower order perturbation of the Lane-Emden critical system 
\[
-\Delta u =  |v|^{p-1}v, \quad -\Delta v = |u|^{q-1}u \quad \text{in} \ \ \Omega, \quad u=v=0 \quad \partial \Omega,
\]
as \eqref{eq:BN} is a lower order perturbation of the critical Lane-Emden equation
\[
-\Delta u = u^{2^*-1}, \ \ u>0 \ \ \text{in} \ \ \Omega, \ \ u =0 \ \ \text{on} \ \ \partial \Omega.
\]
Moreover, condition \eqref{rscondition} on $(r,s)$ for \eqref{sist} corresponds to condition $1 \leq t < 2^*-1$ for \eqref{eq:BN}.

The main result proved in this paper reads as follows.
\begin{theorem}\label{theo1}
Let $\lambda>0$, $\mu>0$, assume \eqref{rscondition} and in case $rs=1$ also suppose that $\lambda \mu^r$ is suitably small. If $N\geq 4$ or, $N=3$ and $p\leq7/2$ or $p\geq8$, then \eqref{sist} has a classical positive solution.
\end{theorem}
The precise condition on the size of $\lambda\mu^r$ (for the case with $rs=1$) is specified at \eqref{mulambda} and \eqref{lambdamu} ahead. Actually, such condition appeared before in \cite{EdJe-djairo} and corresponds to the hypothesis $\lambda<\lambda_1$ for equation \eqref{eq:BNlinear}. Moreover, as proved in \cite[Theorem 4.2]{van1992variational}, if such condition is not verified, then \eqref{sist} may have no positive solution in starshaped domains. Also observe that, in case of $N=3$, $(7/2, 8)$ and $(8,7/2)$ are symmetric points on the critical hyperbola \eqref{pqHC}.

We call the attention to the fact that, when $\lambda=\mu$, $r= s$, $p=q$, any solution of \eqref{sist} is such that $u=v$ (see \cite[Example 4.3]{Ederson-Gabrielle-Nicola}), which makes \eqref{sist} and \eqref{eq:BN} to be equivalent in this case. With this in mind, for $N=3$, the so-called critical dimension for \eqref{eq:BN}, we prove the existence of solutions for $(p,q)$ lying on some parts of the critical hyperbola (even if $r=s=1$), which brings new results when comparing to \cite[Theorem 2]{HMV}, where the case $N= 3$ is not considered. Indeed, when setting side by side our results with \cite[Theorem 2]{HMV}, our contribution is threefold: we treat the case $N=3$; for $N=4$ we do not impose $p\neq 2$ or $p \neq 5$; for $N\geq 3$ we consider the natural 1-homogenous ($rs =1$) or superlinear ($rs>1$) perturbations, while \cite[Theorem 2]{HMV} is restricted to the case with $r=s=1$, $p>1$ and $q>1$. In particular, for $N\geq 4$, we cover all the points $(p,q)$ on the critical hyperbola, which includes points with $p<1$ or $q<1$ for $N >4$. Figure \ref{fig1} ahead illustrates the existence result given by Theorem \ref{theo1} for $N\geq 4$.
 
 \vspace{-.1cm}

 \begin{figure}[h]
\begin{minipage}{8cm}
\begin{center}
     \includegraphics[scale=.9]{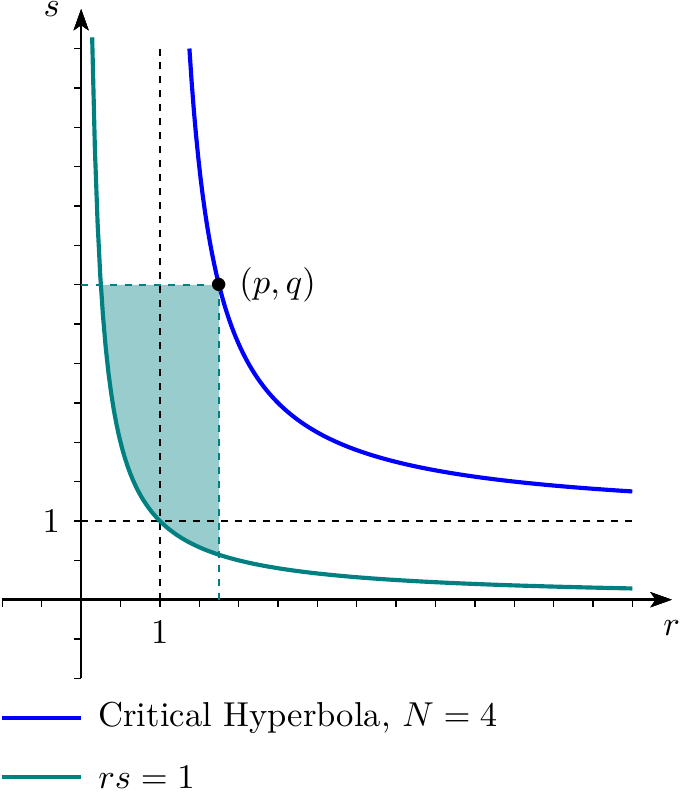}
          \end{center}
\end{minipage}
\begin{minipage}{8cm}
\begin{center}
     \includegraphics[scale=.9]{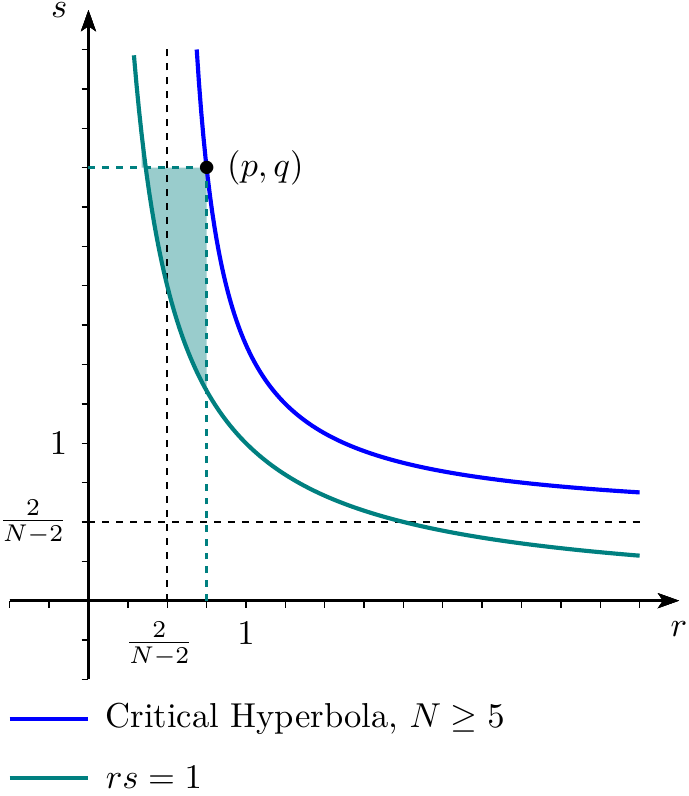}
     \end{center}
\end{minipage}
\caption{\small{Given any $(p,q)$ on the critical hyperbola, any $(r,s)$ satisfying \eqref{rscondition} is \\ admissible for finding a positive solution to \eqref{sist}.}} \label{fig1}
\end{figure}

 \vspace{-.1cm}
Our paper may also serve as motivation for future investigation. In view of the results in \cite[Theorem 1.1]{Pistoia-Kim}, that consider $r=s=1$, it could be interesting to study blowing up phenomena for system \eqref{sist}, with $rs =1$, as $\lambda = \mu \to 0$. 

We recall that in the critical dimension $N=3$  it is not possible to prove the existence of solution for \eqref{eq:BN} in the full range $1\leq t < 5$. Indeed, as in \cite[Corollary 2.3]{BrezisNirenberg1983}, such existence results is proved only for $3 < t < 5$. This motivates to introduce the following definition and open problems.

\begin{defi}\label{def:critical}
For $N=3$, let $(p,q)$ be a point on the critical hyperbola \eqref{pqHC}, $\Omega$ be a bounded regular domain, and $(r,s)$ satisfying \eqref{rscondition}. We say that $(p,q)$ is on a {\emph{Critical Region}} if, for some $\Omega$ and some $(r,s)$, \eqref{sist} has no positive solution for some $\lambda$ and $\mu$ small. On the other hand,  $(p,q)$ is on a {\emph{Noncritical Region}} if for all $\Omega$, all $(r,s)$ satisfying \eqref{rscondition} and $(\lambda, \mu)$ suitably small, then \eqref{sist} has a positive solution.
\end{defi}

\noindent \textbf{Open problems.} 
\begin{enumerate}
\item To find the critical region of the critical hyperbola \eqref{pqHC} for $N=3$. 
\item A simpler problem, but still challenging, is to find the optimal values $7/2 < p_* \leq  p^* < 8$ such that \eqref{sist} has no solution for any $p_* \leq p \leq p^*$ with $r=s=1$, $\lambda=\mu$ small, with $\Omega = B(0,1) \subset \mathbb{R}^3$.
\end{enumerate}

For this second question, due to the results in \cite[Theorem 1.2]{BrezisNirenberg1983} and conditions \eqref{eq:obsN3} ahead in this paper, we know that  $4 \leq p_* \leq 5 \leq p^* \leq 13/2$.  Figure \ref{fig2} illustrates the open problem regarding what should be critical and noncritical regions of the critical hyperbola for $N=3$.

\begin{figure}[h]
\centering
     \includegraphics[scale=.9]{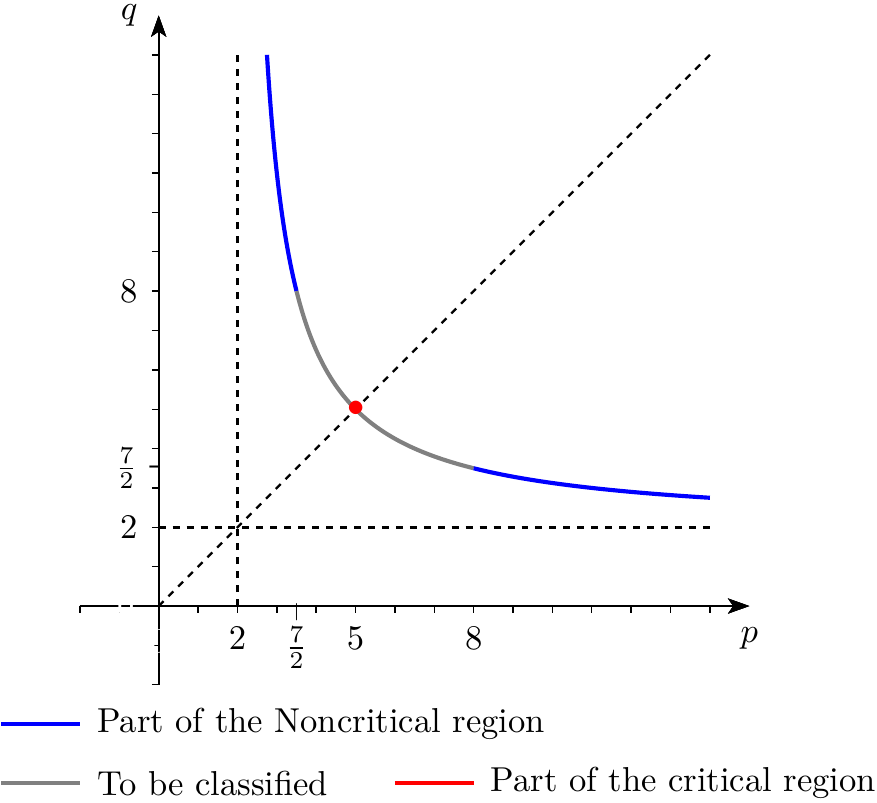}
     \caption{\small{Critical Hyperbola for $N=3$}}\label{fig2}
\end{figure}

Finally, we make a link between critical/noncritical regions of the critical hyperbola associated to Hamiltonian systems and the critical dimensions for the biharmonic operator under Navier boundary conditions. We recall that, according to \cite{van1995}, the dimensions $N=5,6,7$ are named as critical for the study of
\begin{equation}\label{biharmonic}
\Delta^2 u = \mu u + u^{\frac{N+4}{N-4}} \ \ \text{in} \ \ \Omega, \quad u = \Delta u = 0 \ \ \text{on} \ \ \partial \Omega,
\end{equation}
a particular case of \eqref{sist} with $\lambda=0$ and $p=1$; see also \cite{edmunds1990, Bernis_1995} for the case with Dirichlet boundary conditions for the biharmonic and polyharmonic operators, respectively. A first attempt to understand the phenomenon of critical dimension for Hamiltonian systems was presented in \cite{EdJe-djairo}. However, the  asymmetric perturbation in \cite{EdJe-djairo} makes the problem more like a nonlinear version of the biharmonic equation \eqref{biharmonic}, as the counterpart of the $p-$Laplacian version for \eqref{eq:BN}.  In the case with $\lambda>0$ and $\mu>0$ in \eqref{sist}, the natural symmetric perturbation of the critical Lane-Emden system, we recover that the only critical dimension is $N=3$, as it happens to the scalar problem \eqref{eq:BN}, unveiling  the notions of critical and noncritical regions of the critical hyperbola for $N=3$.

This paper is organized as follows. In Section \ref{sec:variational} we present the variational approach to treat the system, writing \eqref{sist} as the fourth order equations \eqref{prob} or \eqref{probp}. We define the energy functionals associated to these equations, show that they have the mountain pass geometry and present an upper bound for their mountain pass levels. Section \ref{sec:PS} is devoted to localize the range where such functionals satisfy the $(PS)_c$-condition and to the proof of Theorem \ref{theo1}. Finally, we accommodate in Section \ref{app} some technical estimates which are crucial for the variational treatment.

\section{Variational approach and Mountain Pass Geometry}\label{sec:variational}

To deal with \eqref{sist}, following the same approach as in \cite{EdJe-djairo}, define
\begin{equation}\label{fglambda}
\begin{array}{ccc}
 f_\lambda(t)=\lambda |t|^{r-1}t+|t|^{p-1}t, \ \ \overline{F}_\lambda(t)=\displaystyle \int_0^t f_\lambda^{-1}(t) dt,\\
  g_\mu(t)=\mu |t|^{s-1}t+|t|^{q-1}t, \ \ \overline{G}_\mu(t)=\displaystyle \int_0^t g_\mu^{-1}(t) dt,
\end{array}
\end{equation}
and rewrite \eqref{sist} as one of the fourth-order equations under Navier boundary conditions
\begin{equation}\tag{P}\label{prob}
\left\{
\begin{array}{lll}
\Delta(f_\lambda^{-1}(\Delta u)) = \mu |u|^{s-1}u+ |u|^{q-1}u \text{ in } \Omega,\\
u, \Delta u=0 \text{ on } \partial\Omega,
\end{array}\right.\end{equation}
\begin{equation}\tag{P'}\label{probp}
 \left\{
\begin{array}{lll}
\Delta(g_\mu^{-1}(\Delta v)) = \lambda |v|^{r-1}v+ |v|^{p-1}v \text{ in } \Omega,\\
v, \Delta v=0 \text{ on } \partial\Omega.
\end{array}\right.\end{equation}

Associated with \eqref{prob} and \eqref{probp}, we consider the $C^1(E_p,\R)$ and $C^1(E_q,\R)$ functionals
\begin{equation}\label{functional}
I_F(u)=\displaystyle \int_{\Omega} \overline{F}_\lambda(\Delta u)dx-\frac{\mu}{s+1} \displaystyle \int_{\Omega} |u|^{s+1}dx -\frac{1}{q+1} \displaystyle \int_{\Omega} |u|^{q+1}dx,
\end{equation}
\begin{equation}\label{auxfunctional}
I_G(u)=\displaystyle \int_{\Omega} \overline{G}_\mu(\Delta v)dx-\frac{\lambda}{r+1} \displaystyle \int_{\Omega} |v|^{r+1}dx -\frac{1}{p+1} \displaystyle \int_{\Omega} |v|^{p+1}dx,
\end{equation}
where $E_t:=W^{2,\frac{t+1}{t}}(\Omega)\cap W^{1,\frac{t+1}{t}}_0(\Omega)$ is endowed with the norm $\Vert u\Vert =|\Delta u|_\frac{t+1}{t}$. Throughout in this paper $|w|_{\theta}$ stand for the $L^{\theta}(\Omega)$-norm of $w$. 

The variational treatment of \eqref{sist} given by studying \eqref{prob} or \eqref{probp} is usually called reduction by inversion. This idea has been used by P. L. Lions \cite{lions} and in several other papers, as for example in \cite{ClementMitidieri, ClementFelmerMitidieri, Hulshof-VanderVorst, EdersonPortugaliae}. Here, since the functions $f_\lambda$ and $g_{\mu}$ are not pure power, and due to the critical growth nature of \eqref{sist}, we prove in Section \ref{sectiontec} some sharp estimates on $f_\lambda$, whose corresponding versions to $g_{\mu}$ also hold. In order to capture in this inversion the contribution of the term $\lambda |u|^{r-1}u$, to downsize the Mountain Pass level, we compute at Lemma \ref{estimateFtv} some integrals on rings involving the ground state solutions of the Lane-Emden critical system on $\R^N$, where terms associated to $u \mapsto \lambda |u|^{r-1}u$ are dominant.

\begin{defi}
We say that $u \in E_p$ is a weak solution of \eqref{prob} iff $I_F'(u)=0$. A function $u \in C^2(\overline{\Omega})$ such that $f_{\lambda}^{-1}(\Delta u) \in C^2(\Omega)$ is a classical solution of \eqref{prob} iff satisfies \eqref{prob} pointwise. Similarly, we define weak and classical solutions of \eqref{probp}. Moreover, $(u,v)$ is a classical solution of \eqref{sist}, iff $u, v \in C(\overline{\Omega}) \cap C^2(\Omega)$ satisfy \eqref{sist} pointwise.
\end{defi}
\begin{lem}\label{regularity Lemma}
If $u$ is a weak solution of \eqref{prob}, then it is a classical solution of \eqref{prob}. The converse is also true. Moreover, $u$ is a classical solution of \eqref{prob}, iff $(u,v)$ is a classical solution of \eqref{sist}, with $v = f_{\lambda}^{-1}(-\Delta u)$.
\end{lem}
\begin{proof}
We can mimic the proof of  \cite[Lemma 1]{EdJe-djairo}, which is based on the arguments in \cite[Section 4]{ederson-JMAA} and \cite[Section 3]{hulshofvandervorst-93}.
\end{proof}

Next we show that the functionals $I_F$ and $I_G$ have the Mountain Pass geometry and obtain upper bounds for their Mountain Pass levels. For the cases with $rs=1$ we introduce the conditions
\begin{equation}\label{mulambda}
\displaystyle\lambda^{1/r}\mu\leq  \frac{(2|\Omega|)^{\frac{r-p}{r(p+1)}}}{2^\frac{r+1}{r}}\mathcal{C}^{\frac{r+1}{r}}_{r,\Omega},
\end{equation} 
\begin{equation}\label{lambdamu}
\displaystyle\lambda\mu^{1/s}\leq  \frac{(2|\Omega|)^{\frac{s-q}{s(q+1)}}}{2^\frac{s+1}{s}}\mathcal{C}^{\frac{s+1}{s}}_{s,\Omega},
\end{equation}
on the size of $(\lambda, \mu)$, where
\begin{equation}
\mathcal{C}_{r,\Omega} = \inf\{\Vert u\Vert ;\ u\in E_p \text{ and } |u|_\frac{r+1}{r}=1\}, \ \ \mathcal{C}_{s,\Omega} = \inf\{\Vert v\Vert ;\ v\in E_q \text{ and } |v|_\frac{s+1}{s}=1\}.
\end{equation}

\begin{obs}
Conditions \eqref{mulambda} and \eqref{lambdamu} for the case with $rs=1$ are natural and correspond to the hypothesis on $\lambda$ and $\mu$ in \cite[Theorem 2]{HMV} to treat \eqref{sist} with $r= s=1$, and to the hypothesis $\lambda < \lambda_1$ in \cite{BrezisNirenberg1983} to study \eqref{eq:BNlinear}.
\end{obs}

\begin{prop}\label{prop-mpgeometry}
Let $(p,q)$ and $(r,s)$ be as in \eqref{pqHC} and \eqref{rscondition}.  
\begin{enumerate}[-]
\item Then $I_F$ has the Mountain Pass geometry with a local minimum at zero, under the additional condition \eqref{mulambda} when $rs=1$.
\item Then $I_G$ has the Mountain Pass geometry with a local minimum at zero, under the additional condition \eqref{lambdamu} when $rs=1$.
\end{enumerate}
\end{prop}

\begin{proof}
Observe that $I_F(0) = 0$ and, from \eqref{ineq1'},
\begin{equation}
\begin{array}{ccc}
I_F(u)\leq\displaystyle \frac{p}{p+1} \Vert u\Vert ^{\frac{p+1}{p}}-\frac{\mu}{s+1} |u|_{s+1}^{s+1} -\frac{1}{q+1} |u|_{q+1}^{q+1}, \ \ \forall \, u \in E_p.
\end{array}
\end{equation}
Then, $I_F(tu)\rightarrow -\infty$ when $t\rightarrow \infty$ and $u \neq 0$.

On the other hand, by Lemma \ref{ineqFtrp},
\begin{multline}
I_F(u) = \displaystyle \int_{|\Delta u|\leq2\lambda^\frac{p}{p-r}}    \overline{F}_\lambda(\Delta u)dx+\displaystyle \int_{|\Delta u|>2\lambda^\frac{p}{p-r}} \overline{F}_\lambda(\Delta u)dx-\frac{\mu}{s+1} |u|_{s+1}^{s+1} -\frac{1}{q+1} |u|_{q+1}^{q+1}\vspace{5pt} \\
 \geq \displaystyle \frac{1}{2^\frac{r+1}{r}\lambda^{1/r}}\frac{r}{r + 1}  \int_{|\Delta u|\leq2\lambda^\frac{p}{p-r}}   |\Delta u|^\frac{r+1}{r} dx+ \frac{1}{2^\frac{p+1}{p}}\frac{p}{p + 1} \displaystyle \int_{|\Delta u|>2\lambda^\frac{p}{p-r}}   |\Delta u|^\frac{p+1}{p}dx \vspace{5pt}\\
  \quad -\frac{\mu}{s + 1} |u|_{s+1}^{s+1} -\frac{1}{q + 1} |u|_{q+1}^{q+1}.
\end{multline}
By Jensen's inequality, for a nonnegative measurable function $a$ and $\alpha>1$,
\begin{equation}\label{jensen}
\int_\omega (a(t))^\alpha dt\geq |\omega|^{1-\alpha}\left(\int_\omega a(t) dt\right)^\alpha.
\end{equation}
Since $0<r<p$, with $\alpha =\frac{r+1}{r} \frac{p}{p+1}>1$, it follows that
\begin{multline}
I_F(u)  \geq  \displaystyle \frac{(meas(|\Delta u|\leq2\lambda^\frac{p}{p-r}))^{1-\alpha}}{2^\frac{r+1}{r}\lambda^{1/r}}\frac{r}{r+1} \left( \int_{|\Delta u|\leq2\lambda^\frac{p}{p-r}} |\Delta u|^\frac{p+1}{p} dx\right)^\alpha \\
+  \displaystyle  \frac{1}{2^\frac{p+1}{p}}\frac{p}{p + 1} \int_{|\Delta u|>2\lambda^\frac{p}{p-r}} |\Delta u|^\frac{p+1}{p}dx \displaystyle -\frac{\mu}{s+1} |u|_{s+1}^{s+1} -\frac{1}{q+1} |u|_{q+1}^{q+1} \\
\geq \displaystyle \frac{|\Omega|^{\frac{r-p}{r(p+1)}}}{2^\frac{r+1}{r}\lambda^{1/r}}\frac{r}{r+1} \left( \int_{|\Delta u|\leq2\lambda^\frac{p}{p-r}} |\Delta u|^\frac{p+1}{p} dx\right)^\alpha\\
+ \displaystyle  \frac{1}{2^\frac{p+1}{p}}\frac{p}{p + 1} \int_{|\Delta u|>2\lambda^\frac{p}{p-r}} |\Delta u|^\frac{p+1}{p}dx \displaystyle -\frac{\mu}{s+1} |u|_{s+1}^{s+1} -\frac{1}{q+1} |u|_{q+1}^{q+1}.
\end{multline}
For $u\neq0$ such that
\begin{equation}
\displaystyle  \frac{1}{2^\frac{p+1}{p}}\frac{p}{p + 1} \Vert u\Vert ^{-\frac{p-r}{rp}}\geq \frac{|\Omega|^{\frac{r-p}{r(p+1)}}}{2^\frac{r+1}{r}\lambda^{1/r}}\frac{r}{r+1} \ \text{ i.e.} \ \ \
\Vert u\Vert  \leq \left(\frac{p(r+1)}{(p+1)r}\right)^\frac{pr}{p-r}2|\Omega|\lambda^\frac{p}{p-r},
\end{equation}
it follows that
\begin{multline}
I_F(u)\geq\displaystyle\frac{|\Omega|^{\frac{r-p}{r(p+1)}}}{2^\frac{r+1}{r}\lambda^{1/r}}\frac{r}{r+1}\left[ \left( \int_{|\Delta u|\leq2\lambda^\frac{p}{p-r}} |\Delta u|^\frac{p+1}{p} dx\right)^\alpha + \left( \int_{|\Delta u|>2\lambda^\frac{p}{p-r}} |\Delta u|^\frac{p+1}{p} dx\right)^\alpha\right]\\ \displaystyle-\frac{\mu}{s+1} |u|_{s+1}^{s+1} -\frac{1}{q+1} |u|_{q+1}^{q+1}.
\end{multline}
Since $1-\alpha=-\frac{p-r}{r(p+1)}=-\frac{p}{p+1}\frac{p-r}{pr}$ and $(a+b)^\alpha\leq 2^{\alpha-1}(a^\alpha + b^\alpha)$ for $a,b \geq 0$, we infer that
\begin{equation}
I_F(u)\geq\displaystyle \frac{(2|\Omega|)^{\frac{r-p}{r(p+1)}}}{2^\frac{r+1}{r}\lambda^{1/r}}\frac{r}{r+1} \Vert u \Vert^{\frac{r+1}{r}}-\frac{\mu}{s+1} |u|_{s+1}^{s+1} -\frac{1}{q+1} |u|_{q+1}^{q+1}
\end{equation}
for all $u\in E$ such that $\Vert u\Vert \leq  \left(\frac{p(r+1)}{(p+1)r}\right)^\frac{pr}{p-r}2|\Omega|\lambda^\frac{p}{p-r}$. 

If $\frac{r+1}{r}< s+1$, i.e. $ \frac{1}{r}<s$, and since $s<q$, it follows that $I_F$ has the Mountain Pass geometry with a local minimum at zero.\\

On the other hand, if $\frac{r+1}{r}= s+1$, i.e. $ \frac{1}{r}=s$, for $u\neq 0$, we infer that
\[
\begin{array}{rcl}
(s+1)I_F(u) &\geq& \frac{(2|\Omega|)^{\frac{r-p}{r(p+1)}}}{2^\frac{r+1}{r}\lambda^{1/r}}\Vert u\Vert ^{\frac{r+1}{r}}-\mu   |u|^{\frac{r+1}{r}}_{\frac{r+1}{r}} -\frac{s+1}{q+1} |u|_{q+1}^{q+1} \vspace{10pt} \\
& = &\left( \frac{(2|\Omega|)^{\frac{r-p}{r(p+1)}}}{2^\frac{r+1}{r}\lambda^{1/r}} - \mu  |u|^{\frac{r+1}{r}}_{\frac{r+1}{r}} \Vert u\Vert ^{-\frac{r+1}{r}} \right)\Vert u\Vert ^{\frac{r+1}{r}}-\frac{s+1}{q+1} |u|_{q+1}^{q+1}\vspace{10pt} \\
& \geq & \left( \frac{(2|\Omega|)^{\frac{r-p}{r(p+1)}}}{2^\frac{r+1}{r}\lambda^{1/r}} - \mu \frac{1}{\mathcal{C}_{r,\Omega}^{\frac{r+1}{r}}} \right)\Vert u\Vert ^{\frac{r+1}{r}}-\frac{s+1}{q+1} |u|_{q+1}^{q+1}
\end{array}
\]
and \eqref{mulambda} gives $\frac{(2|\Omega|)^{\frac{r-p}{r(p+1)}}}{2^\frac{r+1}{r}\lambda^{1/r}} - \mu \frac{1}{\mathcal{C}_{r,\Omega}^{\frac{r+1}{r}}}>0$, and again $I_F$ has the Mountain Pass geometry around zero, since $\frac{1}{r} = s < q$.
\end{proof}

Let $S$ be the Sobolev constant for the embedding $E_p\hookrightarrow L^{q+1}(\Omega)$, namely
\[
S = \inf_{u\in E_p, |u|_{q+1} = 1} \|u\|.
\]

\begin{prop}\label{prop-mplevel}
Suppose \eqref{pqHC} and \eqref{rscondition}, $\mu>0, \ \lambda>0$ and in the case $rs=1$ also assume \eqref{mulambda}. If $N\geq 4$ and $p\leq (N+2)/(N-2)$, or $N=3$ and $p\leq7/2$, then the mountain pass level $c_F$ of the functional $I_F$ is such that $c_F \in (0,\frac{2}{N} S^{\frac{pN}{2(p+1)}})$.
\end{prop}
\begin{proof}
See Section \ref{sectionMP}.
\end{proof}

\section{(PS)$_c$ condition} \label{sec:PS}

When treating \eqref{prob}, the main difficulty is the lack of compactness for the embedding $E_p\hookrightarrow L^{q+1}(\Omega)$. Here we localize the levels $c$ for which the $(PS)_c$ condition holds. Throughout this section \eqref{pqHC}, \eqref{rscondition}, $\mu,\lambda>0$ are assumed, and the main results is the following.
\begin{prop}\label{propcompacidade}
$I_F$ satisfies the $(PS)_c$ condition for all $c<\frac{2}{N} S^{\frac{pN}{2(p+1)}}$.
\end{prop}

We split the proof of this proposition in several lemmas.

\begin{lem}\label{psltd}
Every (PS) sequence for $I_F$ is bounded
\end{lem}
\begin{proof}
Let $(u_n)$ be a (PS) sequence for $I_F$. So, using Corollary \ref{corFs} and Lemma \ref{lema2notas}, there exists $c\in \R$ and a positive sequence $(\epsilon_n)$ with $\epsilon_n \rightarrow 0$ such that 
\begin{multline}
   \displaystyle \epsilon_n \Vert u_n\Vert + c \geq I_F(u_n)- \frac{I_F'(u_n)u_n}{s+1} 
   =\displaystyle \int_{\Omega}  \overline{F}_\lambda(\Delta u_n)-\frac{f_\lambda^{-1}(\Delta u_n) \Delta u_n}{s+1}   dx +\frac{(q-s)|u_n|_{q+1}^{q+1}}{(s+1)(q+1)} \\
  \geq \displaystyle \int_{|\Delta u_n|\geq2\lambda^\frac{p}{p-r}}    \overline{F}_\lambda(\Delta u_n)-\frac{f_\lambda^{-1}(\Delta u_n) \Delta u_n }{s + 1}  dx 
  \geq \tau \displaystyle \int_{|\Delta u_n|\geq 2\lambda^\frac{p}{p-r}} |\Delta u_n|	^{\frac{p+1}{p}}dx \\
  \geq \tau \left(  \displaystyle \int_{\Omega} |\Delta u_n|^{\frac{p+1}{p}}dx - \displaystyle \int_{|\Delta u_n|\leq 2\lambda^\frac{p}{p-r}} |\Delta u_n|^{\frac{p+1}{p}}dx\right)\geq \tau\Vert u_n\Vert ^{\frac{p+1}{p}}-2^{\frac{p+1}{p}}\lambda^\frac{p+1}{p-r}|\Omega|,
\end{multline}
which implies the boundedness of $(\Vert u_n\Vert )$.
\end{proof}

To localize the levels where $I_F$ satisfies the $(PS)$ condition the following result, due to P.-L. Lions, is necessary.
\begin{lem}\label{convseqlim}
Given a bounded sequence $(u_n)$ in $E_p$, there exists a subsequence, also denoted here by $(u_n)$, such that:
\begin{itemize}
\item[(i)] $u_n \rightharpoonup u$ in $E_p$.
\item[(ii)] $u_n \rightarrow u$ a.e. in $\Omega$ and in $L^{\theta}(\Omega)$, for all $1\leq \theta < q+1$.
\item[(iii)] $\left| \Delta u_n\right|^{\frac{p+1}{p}} \stackrel{*}{\rightharpoonup} \gamma$ in the sense of measures on $\overline{\Omega}$.
\item[(iv)] $\left| u_n \right|^{q+1} \stackrel{*}{\rightharpoonup} \nu$ in the sense of measures on $\overline{\Omega}$.
\item[(v)] There exist an at most countable index set $J$, a family of points $\{ x_j: j \in J \} \subset \overline{\Omega}$ and two sequences $\{ \nu_j: j \in J \} ,\{ \gamma_j: j \in J\} \subset (0, +\infty)$ such that:
\[
\displaystyle{ \nu = \left| u \right|^{q+1} + \sum_{j \in J} \nu_j \delta_{x_j}, \,\, \gamma \geq \left| \Delta u \right|^{\frac{p+1}{p}} + \sum_{j\in J} \gamma_j \delta_{x_j}},
\]
\[
\displaystyle {S \, \nu_j^{\frac{p+1}{p}\frac{1}{q+1}} \leq \gamma_j \,\, \hbox{for all}\,\, j \in J, \,\, \hbox{in particular}\,\,\sum_{j \in J} \nu_j^{\frac{p+1}{p}\frac{1}{q+1}} < + \infty}.
\]
\item[(vi)] $\nabla u_n \rightharpoonup \nabla u$ in $\left( W^{1, \frac{p+1}{p}}(\Omega) \right)^N$.
\item[(vii)] $\nabla u_n \rightarrow \nabla u$ a.e. in $\Omega$ and in $\left( L^{\sigma}(\Omega) \right)^N$, for all $1 \leq \sigma < \sigma^*$, with $\sigma^* > \frac{p+1}{p}$ depending on the critical Sobolev embedding of $W^{1, \frac{p+1}{p}}(\Omega)$.
\end{itemize}
\end{lem}
\begin{proof}
See \cite[Lemma I.1]{lions} or \cite[Lemma 3.3]{ederson}.
\end{proof}
An improvement of the previous lemma is given next.
\begin{lem}\label{Jfinito}
If $(u_n)$ is a (PS)-sequence for $I_F$, then there exists a subsequence, for short also denoted by $(u_n)$, satisfying \rm{(i)-(vii)} from Lemma \ref{convseqlim} with the additional fact that $J$ is at most finite.
\end{lem}
\begin{proof}
Let $x_j \in \overline{\Omega}$ be a point in the singular support of $\mu$ and $\nu$. Let $\zeta \in C^{\infty}_{c}(\R^{N} )$ such that $0\leq \zeta \leq 1$, $\zeta\equiv1$ in $B(0,1)$ and $supp(\zeta) \subset B(0,2)$. Moreover for each $\theta >0$ define $\zeta_\theta(x) := \zeta(\frac{x-x_j}{\theta})$.
So there exists constants $c_1$ and $c_2$ independent of  $\theta$ such that
$$ |\nabla \zeta_\theta (x)| \leq \frac{c_1}{\theta}, \  |\Delta \zeta_\theta (x)| \leq \frac{c_2}{\theta^2}, \ \forall\, x \in \R^N.$$
By  \cite[Lemma 3.4 ]{ederson}, $u_n \zeta_\theta \in E_p, \ \forall\, n \in \N \ and \ \theta>0$. Fixing $\theta >0$, since $(\zeta_\theta u_n)$ is bounded in $E_p$, $ \langle I_F'(u_n),\zeta_\theta u_n\rangle=o(1)$ that is
\begin{equation}\label{notas 1}
\begin{array}{lll}
o(1)&=\displaystyle \int_{\overline{\Omega}} f_\lambda^{-1}(\Delta u_n)\Delta u_n \zeta_\theta dx - \displaystyle \int_{\overline{\Omega}} |u_n|^{q+1} \zeta_\theta dx - \mu \displaystyle \int_{\overline{\Omega}} |u_n|^{s+1}\zeta_\theta dx\\
 &\quad+ \displaystyle \int_{\overline{\Omega}} f_\lambda^{-1}(\Delta u_n) u_n \Delta \zeta_\theta dx + 2 \displaystyle \int_{\overline{\Omega}} f_\lambda^{-1}(\Delta u_n) \nabla u_n \nabla \zeta_\theta dx.
\end{array}
\end{equation}
On the other hand, $\zeta_{\theta} (x) \xrightarrow{\theta \rightarrow 0} \delta_{x_j} (x),\ \forall\, x \in \Omega$.
So from $u_n \rightharpoonup u$ in $E_p$ and $E_p\subset\subset L^{s+1}(\Omega)$
\begin{equation}
\displaystyle \int_{\overline{\Omega}} |u_n|^{s+1}\zeta_\theta dx \xrightarrow{n \rightarrow \infty}  \displaystyle \int_{\overline{\Omega}} |u|^{s+1} \zeta_\theta dx \xrightarrow{\theta \rightarrow 0} 0.
\end{equation}
Now, by \eqref{ineqinv} and Hölder inequality $\left( \frac{1}{p+1}+ \frac{1}{q+1} + \frac{2}{N}=1 \right)$, there exists $C>0$ independent of $n$ and $\theta$ such that
\begin{equation}
\displaystyle \left| \int_{\overline{\Omega}} f_\lambda^{-1}(\Delta u_n) u_n \Delta \zeta_\theta dx \right| \leq \! \int_{\Omega} |\Delta u_n|^{1/p} |u_n| |\Delta \zeta_\theta| dx \leq C \!\left( \!\displaystyle \int_{\overline{\Omega}} |u_n|^{q+1} \left|\Delta \zeta\left(\frac{x-x_j}{\theta}\right)\right|^{\frac{q+1}{2}} \!dx\right)^{\frac{1}{q+1}} \!\!\!\!.
\end{equation}
From the definition of the weak* convergence
\begin{equation}
\displaystyle \int_{\overline{\Omega}} |u_n|^{q+1}  \left|\Delta \zeta\left( \frac{x-x_j}{\theta}\right)\right|^{\frac{q+1}{2}}  dx \xrightarrow{n \rightarrow \infty} \displaystyle \int_{\overline{\Omega}}   \left|\Delta\zeta\left( \frac{x-x_j}{\theta} \right)\right|^{\frac{q+1}{2}}  d\nu,
\end{equation}
and since $|\Delta \zeta (\frac{x-x_j}{\theta})|^{(\frac{q+1}{2})} \xrightarrow{\theta \rightarrow 0} 0\ \forall\, x \in \Omega$, by the Lebesgue dominated convergence theorem, 
\begin{equation}
\int_{\overline{\Omega}} \left|\Delta \zeta \left(\frac{x-x_j}{\theta}\right)\right|^{\frac{q+1}{2}} d\nu \xrightarrow{\theta \rightarrow 0} 0.
\end{equation}
We also have
\begin{multline}
\left| \displaystyle \int_{\overline{\Omega}} f_\lambda^{-1}(\Delta u_n) \nabla u_n \nabla \zeta_\theta dx \right| \leq C \left(\displaystyle \int_{\overline{\Omega}}  \left(\frac{1}{\theta} \left| \nabla\zeta\left(\frac{x-x_j}{\theta}\right)\right| |\nabla u_n|\right)^{\frac{p+1}{p}}dx\right)^{\frac{p}{p+1}} \ \ \text{and}\\
\displaystyle \int_{\overline{\Omega}} \!\left(\frac{1}{\theta} \left| \nabla\zeta\left(\frac{x-x_j}{\theta}\right)\right| |\nabla u_n|\right)^{\frac{p+1}{p}}\!\!dx \xrightarrow{n \rightarrow \infty} \displaystyle \int_{\overline{\Omega}} \!\left(\frac{1}{\theta} \left| \nabla\zeta\left(\frac{x-x_j}{\theta}\right)\right| |\nabla u|\right)^{\frac{p+1}{p}}\!\!dx= O(\theta^{N-\frac{p+1}{p}}).
\end{multline}

Given $\epsilon>0$ let $M(\epsilon)>0$ be such that $f_\lambda^{-1}(t)t\geq \frac{1}{1+\epsilon}|t|^\frac{p+1}{p}$ for all $|t|\geq M(\epsilon)$. Then define
\begin{equation}\label{AnBn}
A_n:=\{x\in B(x_j,2\theta)\cap \overline{\Omega}; |\Delta u_n(x)|\geq M(\epsilon)\}, B_n:= (\overline{\Omega}\cap B(x_j,2\theta))\backslash A_n.
\end{equation}
Then, 
\begin{multline}\label{notas 7}
\displaystyle \int_{\overline{\Omega}} f_\lambda^{-1}(\Delta u_n)\Delta u_n \zeta_\theta dx =\displaystyle \int_{A_n} f_\lambda^{-1}(\Delta u_n)\Delta u_n \zeta_\theta dx+\displaystyle \int_{B_n}f_\lambda^{-1}(\Delta u_n)\Delta u_n \zeta_\theta dx \\
 \geq\frac{1}{1+\epsilon} \displaystyle \int_{\overline{\Omega}} |\Delta u_n|^{\frac{p+1}{p}} \zeta_\theta dx + \displaystyle \int_{B_n} (f_\lambda^{-1}(\Delta u_n)\Delta u_n- \frac{1}{1+\epsilon} |\Delta u_n|^{\frac{p+1}{p}}) \zeta_\theta dx\\
 =\frac{1}{1+\epsilon} \displaystyle \int_{\overline{\Omega}} |\Delta u_n|^{\frac{p+1}{p}}\zeta_\theta dx+ \displaystyle \int_{B_n} (f_\lambda^{-1}(\Delta u_n)\Delta u_n- \frac{1}{1+\epsilon} |\Delta u_n|^{\frac{p+1}{p}}) \zeta_\theta dx\longrightarrow \frac{1}{1+\epsilon}\gamma_j
\end{multline}
by taking the limit as $n \rightarrow \infty$ and after as $\theta \rightarrow 0$, because
\begin{multline}
\lim_{\theta\rightarrow 0} \limsup_{n\rightarrow \infty} \left|\displaystyle \int_{B_n} (f_\lambda^{-1}(\Delta u_n)\Delta u_n- \frac{1}{1+\epsilon} |\Delta u_n|^{\frac{p+1}{p}}) \zeta_\theta dx \right|\\ \leq \lim_{\theta\rightarrow 0} \limsup_{n\rightarrow \infty} \displaystyle \int_{B_n} (f_\lambda^{-1}(M)M+\frac{1}{1+\epsilon} M^{\frac{p+1}{p}}) \zeta_\theta dx  = 0.
\end{multline}
Then, from all the above estimates ranging from \eqref{notas 1} to \eqref{notas 7}, we infer that
\begin{equation}
0= \lim_{\theta\rightarrow 0} lim_{n\rightarrow \infty} \langle I_F'(u_n),u_n\zeta_\theta\rangle \geq \frac{\gamma_j}{1+\epsilon}- \nu_j
\end{equation}
which implies that $\nu_j \geq \frac{\gamma_j}{1+\epsilon}$ for all $\epsilon>0$, and hence $\nu_j \geq \gamma_j$.
In contrast, since $0\leq f_\lambda^{-1}(\Delta u_n)\Delta u_n\leq |\Delta u_n|^{\frac{p+1}{p}}$, it follows from all the above estimates ranging from \eqref{notas 1} to \eqref{notas 7} that
\begin{equation}
0= \lim_{\theta\rightarrow 0} \lim_{n\rightarrow \infty} \langle I_F'(u_n),u_n\zeta_\theta\rangle \leq  \gamma_j- \nu_j,
\end{equation}
which implies 
\begin{equation}\label{gamma=nu}
\gamma_j = \nu_j. 
\end{equation}
Then, from  Lemma \ref{convseqlim},
 $\nu_j \geq S\nu_j^{\frac{p+1}{p}\frac{1}{q+1}}$ and so 
 \begin{equation}\label{gammamaior}
 \nu_j \geq S^{\frac{pN}{2(p+1)}},
 \end{equation}
since $\frac{pN}{2(p+1)}=\left(1-\frac{1}{q+1}\frac{p+1}{p}\right)^{-1}$ and $\gamma_j>0$. Combining this with $\sum_{j\in J} \nu_j^{\frac{p+1}{p}\frac{1}{q+1}}<+\infty$, it follows that $J$ is at most finite.
\end{proof}
\begin{lem}\label{elves1}
Given a bounded sequence $(u_n)$ in $E_p$, $K \subset\subset \Omega \backslash\{ x_j: j \in J \}$ with $\{ x_j: j \in J \}$ from Lemma \ref{convseqlim}, then $u_n \rightarrow u$ in $L^{q+1}(K)$, up to a subsequence.
\end{lem}

\begin{proof} See the proof of \cite[Lemma 3.6]{ederson}.
\end{proof}

\begin{lem}\label{ineqnuk}
If $(u_n)$ is a (PS)-sequence for $I_F$, and $\{ x_j: j \in J \}$ from Lemma \ref{convseqlim}, then for every $j\in J$, up to a subsequence,
\begin{equation}\label{ps-1nuk}
\lim_{n\rightarrow \infty} \int_{\overline{\Omega}} \left[\overline{F}_\lambda(\Delta u_n)-\frac{1}{s+1} f_\lambda^{-1}(\Delta u_n) \Delta u_n \right] dx \geq \frac{ps-1}{(p+1)(s+1)} \gamma_j.
\end{equation}
\end{lem}

\begin{proof}
Consider the even function $H_p(t):=|t|^{-\frac{p+1}{p}}\left(\overline{F}_\lambda(t)-\frac{1}{s+1}f_\lambda^{-1}(t)t\right)$. By  Lemma \ref{Ffeq},
\begin{equation}
\lim_{t\rightarrow \infty} H_p (t)= \lim_{t\rightarrow \infty}\left[ \frac{ps-1}{(p+1)(s+1)}\frac{f_\lambda^{-1}(t)}{t^{1/p}}-\frac{p-r}{p+1}\frac{\lambda}{r+1}\frac{|f^{-1}_\lambda(t)|^{r+1}}{t^\frac{p+1}{p}}\right]=\frac{ps-1}{(p+1)(s+1)}.
\end{equation}

Then, given $\epsilon>0$ small, there exists $t_0>0$ such that
$$H_p (t)> c_{\epsilon} >0 \ \ \ \ \ \text{for all } |t|>t_0,$$
with $c_{\epsilon} = \frac{ps-1}{(p+1)(s+1)}-\epsilon$, which implies that
\begin{equation}\label{ps-1nuk2}
\overline{F}_\lambda(t)-\frac{1}{s+1} f_\lambda^{-1}(t)t > c_{\epsilon}|t|^{\frac{p+1}{p}} \ \ \ \ \ \text{for all } |t|>t_0.
\end{equation}
Let $A_n$, $B_n$ and $\zeta_\theta$ be like in Lemma \ref{Jfinito}, with $A_n$ and $B_n$ associated with $t_0$. So, from \eqref{ps-1nuk2} and Corollary \ref{corFs}, 
\begin{multline}\label{parte1eq}
\displaystyle\int_{\overline{\Omega}} \overline{F}_\lambda(\Delta u_n)-\frac{1}{s+1} f_\lambda^{-1}(\Delta u_n) \Delta u_n dx 
\geq\displaystyle c_{\epsilon} \left(\displaystyle \int_{A_n} |\Delta u_n|^{\frac{p+1}{p}} \zeta_{\theta} dx+
\displaystyle \int_{B_n}|\Delta u_n|^{\frac{p+1}{p}} \zeta_\theta dx\right) \\
+ \displaystyle \int_{B_n}\left[ \overline{F}_\lambda(\Delta u_n)-\frac{1}{s+1} f_\lambda^{-1}(\Delta u_n) \Delta u_n - c_{\epsilon}|\Delta u_n|^{\frac{p+1}{p}} \right]\zeta_\theta dx\\
= \displaystyle c_{\epsilon} \displaystyle \int_{\overline\Omega} |\Delta u_n|^{\frac{p+1}{p}} \zeta_{\theta} dx + \displaystyle \int_{B_n}\left[ \overline{F}_\lambda(\Delta u_n)-\frac{1}{s+1} f_\lambda^{-1}(\Delta u_n) \Delta u_n - c_{\epsilon}|\Delta u_n|^{\frac{p+1}{p}} \right]\zeta_\theta dx.
\end{multline}
On the other hand,
\begin{multline}\label{parte2eq}
\lim_{\theta \to 0}\limsup_{n\rightarrow \infty}\left| \displaystyle \int_{B_n}\left[ \overline{F}_\lambda(\Delta u_n)-\frac{1}{s+1} f_\lambda^{-1}(\Delta u_n) \Delta u_n - c_{\epsilon}|\Delta u_n|^{\frac{p+1}{p}}  \right]\zeta_\theta dx\right|\\ 
\leq \lim_{\theta \to 0}\limsup_{n\rightarrow \infty} \displaystyle \int_{\overline\Omega}\left[ \overline{F}_\lambda(t_0)+\frac{1}{s+1} f_\lambda^{-1}(t_0) t_0 + c_{\epsilon}|t_0|^{\frac{p+1}{p}}  \right]\zeta_\theta dx =0,
\end{multline}
and, from  Lemma \ref{convseqlim},
\begin{equation}\label{parte3eq}
\displaystyle \lim_{\theta \to 0}\limsup_{n\rightarrow \infty}c_{\epsilon} \displaystyle \int_{\overline\Omega} |\Delta u_n|^{\frac{p+1}{p}} \zeta_{\theta} dx\geq c_\epsilon\gamma_k.
\end{equation}
From \eqref{parte1eq}, \eqref{parte2eq},  \eqref{parte3eq} and the arbitrariness of $\epsilon>0$, we get the desired inequality.
\end{proof}
\begin{lem}
If $(u_n)$ is a (PS)-sequence for $I_F$, $K \subset\subset \Omega \backslash\{ x_j: j \in J \}$ with $\{ x_j: j \in J \}$ from Lemma \ref{convseqlim}, then up to a subsequence,
\begin{equation}\label{integrand}
\displaystyle \int_K (f_\lambda^{-1}(\Delta u_n)-f_\lambda^{-1}(\Delta u))(\Delta u_n - \Delta u)dx \xrightarrow{n\rightarrow \infty}0.
\end{equation}
\end{lem}
\begin{proof} Let $\delta = dist(K, \{ x_j: j \in J\})$. For each $\theta \in (0, \delta)$, consider $A_{\theta} = \{ x \in \Omega: dist(x, K)< \theta\}$ and $\xi_\theta \in C_c^{\infty}(\Omega)$, $0\leq \xi_{\theta} \leq 1$, $\xi_{\theta}\equiv 1$ on $A_{\theta/2}$ and $\xi_{\theta} \equiv 0$ on $\Omega \backslash A_{\theta}$. So, by the monotonicity of $f_\lambda^{-1}$,
\begin{multline}\label{integrand1}
0 \leq \displaystyle \int_K (f_\lambda^{-1}(\Delta u_n)-f_\lambda^{-1}(\Delta u))(\Delta u_n - \Delta u)dx \leq \displaystyle \int_\Omega (f_\lambda^{-1}(\Delta u_n)-f_\lambda^{-1}(\Delta u))(\Delta u_n - \Delta u)\xi_\theta dx \\ 
=\displaystyle \int_\Omega f_\lambda^{-1}(\Delta u_n)\Delta u_n \xi_\theta - f_\lambda^{-1} (\Delta u_n) \Delta u \xi_\theta - f_\lambda^{-1}(\Delta u)(\Delta u_n - \Delta u) \xi_\theta dx.
\end{multline}

\noindent
Fixing $\theta>0$, since $I_F'(u_n)\rightarrow 0$ and $(u_n \xi_\theta)$ is bounded in $E$, then $\langle I_F'(u_n), \xi_\theta u\rangle=o(1)$ and $\langle I_F'(u_n), \xi_\theta u_n\rangle=o(1)$ that is
\begin{equation}\label{integrand2}
o(1)\!=\!\!\displaystyle \int_{\overline{\Omega}} f_\lambda^{-1}(\Delta u_n)(\Delta u \xi_\theta  + u \Delta \xi_\theta + 2 \nabla u \nabla \xi_\theta) dx -\!\! \displaystyle \int_{\overline{\Omega}} |u_n|^{q-1}u_n \xi_\theta u dx - \mu\! \!\displaystyle \int_{\overline{\Omega}} |u_n|^{s-1} u_n \xi_\theta u dx
\end{equation}
\begin{equation}\label{integrand3}
o(1)\!=\!\!\displaystyle \int_{\overline{\Omega}} f_\lambda^{-1}(\Delta u_n)(\Delta u_n \xi_\theta  + u_n \Delta \xi_\theta + 2 \nabla u_n \nabla \xi_\theta) dx - \!\displaystyle \int_{\overline{\Omega}} |u_n|^{q+1} \xi_\theta dx - \mu \!\displaystyle \int_{\overline{\Omega}} |u_n|^{s+1} \xi_\theta dx.
\end{equation}
From \eqref{integrand1}, \eqref{integrand2}, \eqref{integrand3}, Lemma \ref{convseqlim}, Lemma \ref{elves1} and \cite[Lemma 2.4]{ederson}, it follows that 
\begin{multline}
0 \leq \displaystyle \int_K (f_\lambda^{-1}(\Delta u_n)-f_\lambda^{-1}(\Delta u))(\Delta u_n - \Delta u)dx 
\leq \displaystyle \int_\Omega f_\lambda^{-1}(\Delta u_n)\Delta\xi_\theta (u_n - u)dx\\
 +2\displaystyle \int_\Omega f_\lambda^{-1}(\Delta u_n)  \nabla \xi_\theta \nabla (u_n -u)dx-\displaystyle \int_\Omega f_\lambda^{-1}(\Delta u) (\Delta u_n - \Delta u)\xi_\theta dx +o(1) \\
 \leq \displaystyle C\left(\int_{A_\theta} |u_n-u|^{q+1}dx\right)^{\frac{1}{q+1}}+C\left(\int_{\Omega} |\nabla u_n-\nabla u|^{\frac{p+1}{p}}dx\right)^{\frac{p}{p+1}}\\
 -\displaystyle \int_\Omega f_\lambda^{-1}(\Delta u) (\Delta u_n - \Delta u)\xi_\theta dx +o(1)= o(1). \qedhere
\end{multline}

\end{proof}
\begin{lem}
If $(u_n)$ is a (PS)-sequence for $I_F$, then $\Delta u_n \xrightarrow{n\rightarrow \infty} \Delta u$ a.e. in $\Omega$, up to a subsequence.
\end{lem}
\begin{proof}
Let $K\subset\subset \Omega\backslash\{x_j\}_{j \in J}$. By the inverse of the Lebesgue dominated convergence theorem, there exists a subsequence of the integrand in \eqref{integrand} that converges a.e. in $K$. Using \cite[Lemma 6]{masomurat} with $$X=\R,\ \beta_n=f_\lambda^{-1}, \ \beta=f_\lambda^{-1}, \ \text{and } \xi_n=\Delta u_n(x)$$
we get $\Delta u_n \rightarrow \Delta u \ a.e. \ in \ K$. Since $K$ is an arbitrary compact subset of $\Omega\backslash\{x_j\}_{j \in J}$, we conclude that $\Delta u_n \rightarrow \Delta u \ a.e. \ in \ \Omega$.
\end{proof}

\begin{lem}
If $(u_n)$ is a (PS)-sequence for $I_F$, then $f_\lambda^{-1}(\Delta u_n) \rightharpoonup f_\lambda^{-1}(\Delta u) \text{ in } L^{p+1}(\Omega)$, up to a subsequence.\end{lem}
\begin{proof}
Since, up to a subsequence,
\begin{equation}
\left\{
\begin{array}{lll}
\Delta u_n \xrightarrow{n\rightarrow \infty} \Delta u \ \ \ \ \text{a.e. in }\Omega, \\ (\Delta u_n)\text{ is bounded in } L^{\frac{p+1}{p}}(\Omega), \text{ and}\qquad\ \\
|f_\lambda^{-1}(\Delta u_n)|\leq |\Delta u_n|^{1/p},
 \end{array}
 \right.
\end{equation}
we infer that $f_\lambda^{-1}(\Delta u_n) \xrightarrow{n\rightarrow \infty} f_\lambda^{-1}(\Delta u)$ a.e. in $\Omega$,  $(f_\lambda^{-1}(\Delta u_n))$ is bounded in $L^{p+1}(\Omega)$,  and hence $f_\lambda^{-1}(\Delta u_n) \rightharpoonup f_\lambda^{-1}(\Delta u) \text{ in } L^{p+1}(\Omega)$.
 \end{proof}
\begin{prop}\label{propfracasol}
If $(u_n)$ is a $(PS)$-sequence for $I_F$, then there exist a subsequence, still denoted by $(u_n)$, such that $u_n\rightharpoonup u$ in $E_p$ and $u$ is a weak solution of \eqref{prob}.
\end{prop}
\begin{proof}
First, since $(u_n)$ is a $(PS)$-sequence to $I_F$, $\langle I_F'(u_n), w \rangle \rightarrow 0$, for all $w \in E_p$. On the other hand, up to a subsequence, 
\begin{equation}
\left\{
\begin{array}{lll}
f_\lambda^{-1}(\Delta u_n) \rightharpoonup f_\lambda^{-1}(\Delta u) \text{ in } L^{p+1}(\Omega),\\
|u_n|^{q-1}u_n \rightharpoonup |u|^{q-1}u \text{ in } L^{\frac{q+1}{q}}(\Omega)\text{ and}\\ |u_n|^{s-1}u_n \rightarrow |u|^{s-1}u \text{ in } L^{\frac{s+1}{s}}(\Omega).
 \end{array}
 \right.
\end{equation}
Thus, for all $w \in E_p$, $\langle I_F'(u_n), w \rangle \rightarrow \langle I_F'(u), w \rangle$. Then $\langle I_F'(u), w \rangle=0$, for all $w \in E_p$, that is, $u$ is a weak solution of \eqref{prob}.
\end{proof}

\begin{proof}[\textbf{Proof of Proposition \ref{propcompacidade}}]
Let $(u_n)$ be a $(PS)_c$ sequence for $I_F$ with  $c<\frac{2}{N} S^{\frac{pN}{2(p+1)}}$.  By contradiction, suppose that $J\neq \emptyset$. We can  suppose the assertions of Lemma \ref{Jfinito}, with $\gamma_j = \nu_j$,  $\nu_j \geq S^{\frac{pN}{2(p+1)}}$, for all $j \in J$, by \eqref{gamma=nu} and \eqref{gammamaior}. Since $(u_n)$ is bounded in $E_p$, $\langle I_F'(u_n), u_n \rangle=o(1)$, and Lemmas \ref{convseqlim} and \ref{ineqnuk}, we infer that
\begin{multline}
c=\lim_{n\rightarrow \infty} I_F(u_n)- \frac{1}{s+1} \langle I_F'(u_n), u_n \rangle \\
= \lim_{n\rightarrow \infty} \displaystyle \int_{\overline{\Omega}} \overline{F}_\lambda(\Delta u_n)-\frac{1}{s+1} f_\lambda^{-1}(\Delta u_n) \Delta u_n dx +\left(\frac{1}{s+1}-\frac{1}{q+1}\right) \displaystyle \int_{\overline{\Omega}} |u_n|^{q+1}dx\\
\geq \displaystyle\left( \frac{ps-1}{(p+1)(s+1)}+\frac{1}{s+1}-\frac{1}{q+1}\right) \nu_j= \frac{2}{N}\nu_j\geq \frac{2}{N} S^{\frac{pN}{2(p+1)}}
\end{multline}
for every $j\in J$, which is a contradiction. Hence, $J=\emptyset$.

Then, from Lemmas \ref{convseqlim} and \ref{elves1}, since $L^{q+1}(\Omega)$ is uniformly convex, $u_n \rightarrow u$ in $L^{q+1}(\Omega)$.

Let $v_n=u_n-u$, thus $v_n\rightharpoonup 0$ in $E_p$, $\Delta v_n \rightarrow 0\ a.e.$ in $\Omega$ and $v_n \rightarrow 0$ in $L^{q+1}(\Omega)$. Since
\begin{equation}
|(a+b)f_\lambda^{-1}(a+b)-af_\lambda^{-1}(a)|\leq |a+b|^{\frac{p+1}{p}}+|a|^{\frac{p+1}{p}} \leq 2^p( |b|^{\frac{p+1}{p}} + |a|^{\frac{p+1}{p}}), \ \ \forall\, a,b \in \mathbb{R},
\end{equation}
from \cite[Theorem 2]{brezislieb}, with $j(t)=tf_\lambda^{-1}(t)$, 
\begin{multline}
\int_\Omega f_\lambda^{-1}(\Delta u_n)\Delta u_n dx = \int_\Omega f_\lambda^{-1}(\Delta u+\Delta v_n) (\Delta u +\Delta v_n) dx \\= \int_\Omega f_\lambda^{-1}(\Delta u)\Delta u + f_\lambda^{-1}(\Delta v_n)\Delta v_n dx+o(1).
\end{multline}
Since $u$ is weak solution of \eqref{prob},
\begin{multline}
o(1) = \langle I_F'(u_n), u_n \rangle = \displaystyle \int_{\Omega} f_\lambda^{-1}(\Delta u_n)\Delta u_n  -  |u_n|^{q+1} - \mu  |u_n|^{s+1} dx\\ =\displaystyle \langle I_F'(u), u \rangle+ \int_\Omega f_\lambda^{-1}(\Delta v_n)\Delta v_n dx+o(1)=\int_\Omega f_\lambda^{-1}(\Delta v_n)\Delta v_n dx+o(1),
\end{multline}
that is, $f_\lambda^{-1}(\Delta v_n)\Delta v_n \rightarrow 0$ in $L^1(\Omega)$. Then, Lemma \ref{ineqftpr} and Jensen's inequality \eqref{jensen} lead to

\begin{multline}
0 \leftarrow \displaystyle \int_\Omega f_\lambda^{-1}(\Delta v_n)\Delta v_n dx \geq \frac{1}{2^{1/p}} \int_{|\Delta v_n| \geq 2\lambda^\frac{p}{p-r}} |\Delta v_n|^{\frac{p+1}{p}} dx + \frac{1}{(2\lambda)^{1/r}} \int_{|\Delta v_n|\leq 2\lambda^\frac{p}{p-r}} |\Delta v_n|^{\frac{r+1}{r}} dx \vspace{5pt} \\ 
\geq\frac{1}{2^{1/p}} \int_{|\Delta v_n| \geq 2\lambda^\frac{p}{p-r}} |\Delta v_n|^{\frac{p+1}{p}} dx + \frac{1}{(2\lambda)^{1/r}}|\Omega|^{1-\alpha} \left(\int_{|\Delta v_n|\leq 2\lambda^\frac{p}{p-r}} |\Delta v_n|^{\frac{p+1}{p}} dx\right)^\alpha,
\end{multline}
with $\alpha = \frac{p}{p+1} \frac{r+1}{r}$. Therefore $|\Delta v_n|^{\frac{p+1}{p}} \to 0$ in $L^1(\Omega)$, that is, $u_n \to u$ in $E_p$. 
\end{proof}

At this point we have all the tools at hand to prove our main result.

\begin{proof}[\textbf{Proof of Theorem \ref{theo1}}]
Suppose, without  loss of generality, that $p\leq q$. The case $q\leq p$ can be handled similarly, by using $I_G$ instead of $I_F$. By Propositions \ref{prop-mpgeometry}, \ref{prop-mplevel}, \ref{propcompacidade} and Lemma \ref{regularity Lemma}, the existence of a classical solution is a direct consequence of the Mountain Pass theorem. 

Next, we prove that any Mountain Pass solution is signed. Let $u$ be a Mountain Pass solution of \eqref{prob}. So, by Lemma \ref{regularity Lemma}, $u\in C^2(\overline{\Omega})$ and $u= 0$ on $\partial \Omega$. Then, by the classical strong maximum principle for second-order elliptic operators, it is enough to show that $\Delta u$ does not change sign in $\Omega$.
By contradiction, suppose that $\Delta u$ changes sign in $\Omega$, and let $\omega$ be the solution of
\begin{equation}\left\{
\begin{array}{crl}
-\Delta \omega = |\Delta u| &\text{in } \Omega,\\
\omega=0 &\text{on } \partial\Omega.
\end{array}\right.
\end{equation}
By the Strong Maximum Principle, $\omega>|u|$ in $\Omega$ and we infer that
\begin{multline} c_F\leq\max_{t\geq0} I_F (t\omega) =\max_{t\geq0} \left\{ \int_\Omega \overline{F}_\lambda(t\Delta \omega)dx - \frac{\mu}{s+1} t^{s+1}\int_\Omega |\omega|^{s+1}dx -\frac{t^{q+1}}{q+1} \int_\Omega |\omega|^{q+1}dx\right\}\\
<\max_{t\geq0} \left\{ \int_\Omega \overline{F}_\lambda(t\Delta u)dx - \frac{\mu}{s+1} t^{s+1}\int_\Omega |u|^{s+1}dx -\frac{t^{q+1}}{q+1} \int_\Omega |u|^{q+1}dx\right\}=\max_{t\geq0}I_F(tu)=c_F,
\end{multline}
which is a contradiction. Hence $\Delta u$ does not change sign in $\Omega$, and therefore, up to multiplication by $-1$, $u>0$ and $-\Delta u>0$ in $\Omega$. Finally, by Lemma \ref{regularity Lemma}, with $v = f^{-1}(-\Delta u)$, $(u, v)$ is a positive classical solution of \eqref{sist}.
\end{proof}

\section{Appendix}\label{app}
\subsection{Some technical properties of the auxiliary functions}\label{sectiontec}
Ahead in this section, where \eqref{rscondition} is assumed, some properties of the functions $f_\lambda^{-1},\ \overline{F}_\lambda,\ g_\mu^{-1}$ and $\overline{G}_\mu$, as defined in \eqref{fglambda}, are given. Indeed, we can consider $f_\lambda^{-1}$ and $\overline{F}_\lambda$ and infer the respective properties for the others. We start by showing some useful inequalities. Observe that
\begin{equation}
t^p<f_\lambda(t) \text{ and } \lambda t^r< f_\lambda(t) \ \ \text{for} \ \ t>0
\end{equation}
and writing  $\tau=f_\lambda(t)$ we get
\begin{equation}\label{ineqinv}
f_\lambda^{-1}(\tau)<\tau^{1/p} \text{ and } f_\lambda^{-1}(\tau)<\frac{\tau^{1/r}}{\lambda^{1/r}} \ \ \text{for} \ \ t>0.
\end{equation}
So,
\begin{equation}\label{ineq1'}
 \overline{F}_\lambda(\tau)\leq \frac{p}{p+1} |\tau|^{\frac{p+1}{p}} \text{ and } \overline{F}_\lambda(\tau)\leq \frac{r}{r+1} \frac{1}{\lambda^{1/r}}|\tau|^{\frac{r+1}{r}}\ \ \forall\, \tau\in \R.
\end{equation}

The next lemmas are used to obtain the geometric condition and upper bounds for the critical level of the  Mountain Pass Theorem for the functionals $I_F$ and $I_G$.

\begin{lem}\label{Ffeq}
$\overline{F}_\lambda(t)= \frac{p}{p+1} f^{-1}_\lambda(t)t-\frac{p-r}{p+1}\frac{\lambda}{r+1}|f^{-1}_\lambda(t)|^{r+1}$ for all $t\in\R$
\end{lem}
\begin{proof}
Set $M(t):=\frac{p}{p+1}f_\lambda^{-1}(t)t-\frac{p-r}{p+1}\frac{\lambda}{r+1}|f^{-1}_\lambda(t)|^{r+1}-\overline{F}_\lambda(t)$. Then $M(0)=0$, $M$ is even, and for $t>0$
\begin{multline}
M'(t)  =\displaystyle\frac{1}{p+1}\left( -f_\lambda^{-1}(t)+\frac{pt}{\lambda r[f_\lambda^{-1}(t)]^{r-1}+p[f_\lambda^{-1}(t)]^{p-1}}-\frac{\lambda(p-r)\left[f^{-1}_\lambda(t)\right]^{r}}{\lambda r[f_\lambda^{-1}(t)]^{r-1}+p[f_\lambda^{-1}(t)]^{p-1}}\right)\\
  \displaystyle=\frac{1}{p+1}\left(\frac{\lambda(p-r)\left[f^{-1}_\lambda(t)\right]^{r}}{\lambda r[f_\lambda^{-1}(t)]^{r-1}+p[f_\lambda^{-1}(t)]^{p-1}}-\frac{\lambda(p-r)\left[f^{-1}_\lambda(t)\right]^{r}}{\lambda r[f_\lambda^{-1}(t)]^{r-1}+p[f_\lambda^{-1}(t)]^{p-1}}\right) = 0,
\end{multline}
which implies the desired identity.
\end{proof}

\begin{lem}\label{Ffeq2}
$\overline{F}_\lambda(t)= \lambda\frac{r}{r+1}|f^{-1}_\lambda(t)|^{r+1}+ \frac{p}{p+1} |f^{-1}_\lambda(t)|^{p+1}$ for all $t\in\R$.
\end{lem}
\begin{proof}
The argument follows as in the proof of the last  Lemma, observing that for all $t>0$
\begin{multline}\frac{d}{dt}\left( \lambda\frac{r}{r+1}|f^{-1}_\lambda(t)|^{r+1}+ \frac{p}{p+1} |f^{-1}_\lambda(t)|^{p+1} \right)= \frac{\lambda r\left[f^{-1}_\lambda(t)\right]^{r}+ p \left[f^{-1}_\lambda(t)\right]^{p}}{\lambda r\left[f^{-1}_\lambda(t)\right]^{r-1}+ p \left[f^{-1}_\lambda(t)\right]^{p-1}}=f_\lambda^{-1}(t),\end{multline}
which coincides with $\frac{d}{dt} \overline{F}_{\lambda}(t)$.
\end{proof}

\begin{cor}\label{corFs}
$\displaystyle\overline{F}_\lambda(t)\geq\frac{f_\lambda^{-1}(t)t}{s+1}$, for all $t\in \mathbb{R}.$
\end{cor}

\begin{proof}
Since $f^{-1}_{\lambda}(t)t=f^{-1}_{\lambda}(t)f_\lambda(f_\lambda^{-1}(t)) = \lambda |f_{\lambda}^{-1}(t)|^{r+1} + |f_{\lambda}^{-1}(t)|^{p+1}$, from Lemma \ref{Ffeq2} and \eqref{rscondition}, for all $t \in \mathbb{R}$,
\[\overline{F}_\lambda (t) -\frac{f_\lambda^{-1}(t)t}{s+1}= \lambda\left(\frac{r}{r+1}-\frac{1}{s+1}\right)|f^{-1}_\lambda(t)|^{r+1}  + \left(\frac{p}{p+1}-\frac{1}{s+1}\right) |f^{-1}_\lambda(t)|^{p+1}\geq 0. \qedhere
\]

\end{proof}

\begin{lem}\label{ineqftpr}
\[f_\lambda^{-1}(t)\geq \left\{ \begin{array}{cl}
\displaystyle\left(\frac{t}{2\lambda}\right)^{1/r},& \forall\ 0<t\leq 2\lambda^\frac{p}
{p-r},\\
\displaystyle\left(\frac{t}{2}\right)^{1/p},& \forall\ t\geq 2\lambda^\frac{p}
{p-r}.
\end{array}\right.
\]
\end{lem}
\begin{proof}
Observe that $2\lambda^\frac{p}{p-r}=\lambda (\lambda^\frac{1}{p-r})^r+(\lambda^\frac{1}{p-r})^p=f_\lambda(\lambda^\frac{1}{p-r})$ and write $z=f_\lambda^{-1}(t)$.

If $t\leq 2\lambda^\frac{p}{p-r}$, applying $f_\lambda^{-1}$ to this inequality, one gets $z\leq f^{-1}_{\lambda}(2\lambda^\frac{p}{p-r}) =\lambda^\frac{1}{p-r}$, that is, $ z^{p}\leq\lambda z^r$, and so $t =z^{p}+\lambda z^r\leq2\lambda z^r$, which implies $\left(\frac{t}{2\lambda}\right)^{1/r}\leq f_\lambda^{-1}(t).$

If $t\geq 2\lambda^\frac{p}{p-r}$, then $z\geq \lambda^\frac{1}{p-r}$ and  $2z^{p}\geq\lambda z^r+z^{p}= t$, which implies $f_\lambda^{-1}(t)\geq \left(t/2\right)^{1/p}$, as desired.
\end{proof}

\begin{lem}\label{ineqFtrp}
\[\overline{F}_\lambda(t)\geq \left\{ \begin{array}{ll}
\displaystyle \frac{r}{r+1}\lambda^{-1/r}\left(\frac{t}{2}\right)^\frac{r+1}{r}+\frac{p}{p+1}\left(\frac{t}{2\lambda}\right)^\frac{p+1}{r},& \forall\ |t|\leq 2\lambda^\frac{p}{p-r},\vspace{5pt}\\
\displaystyle  \frac{r}{r+1}\lambda\left(\frac{t}{2}\right)^\frac{r+1}{p}+\frac{p}{p+1}\left(\frac{t}{2}\right)^\frac{p+1}{p},& \forall\ |t|\geq 2\lambda^\frac{p}{p-r}.
\end{array}\right.
\]

\end{lem}

\begin{proof}
It is a straightforward consequence of Lemmas \ref{Ffeq2} and \ref{ineqftpr}.
\end{proof}

\begin{lem}\label{lema2notas}
For $\tau= \frac{ps-1}{2^{\frac{p+1}{p}}(p+1)(s+1)}$, 
\begin{equation}
\overline{F}_\lambda (t) -\frac{f_\lambda^{-1}(t)t}{s+1} \geq \tau |t|^{\frac{p+1}{p}} \ \ \forall \, |t|\geq 2\lambda^\frac{p}{p-r}.\label{lemanovo2}
\end{equation}
\end{lem}

\begin{proof}
By Lemma \ref{ineqftpr} and the proof of Corollary \ref{corFs}, for all $|t|\geq 2\lambda^\frac{p}{p-r},$
\begin{multline}
\overline{F}_\lambda (t) -\frac{f_\lambda^{-1}(t)t}{s+1}=\lambda\left(\frac{r}{r+1}-\frac{1}{s+1}\right)|f^{-1}_\lambda(t)|^{r+1}  + \left(\frac{p}{p+1}-\frac{1}{s+1}\right) |f^{-1}_\lambda(t)|^{p+1}\\
\geq \frac{ps-1}{(p+1)(s+1)} |f^{-1}_\lambda(t)|^{p+1}\geq\frac{ps-1}{(p+1)(s+1)}\left(\frac{|t|}{2}\right)^\frac{p+1}{p}.\qedhere
\end{multline}
\end{proof}

\begin{lem}\label{fineq}
For all $\alpha,\ \beta \in \mathbb{R}$, there exists $\theta\in(0,\ 1)$ such that
\begin{equation} \displaystyle0 \leq f^{-1}_\lambda(\alpha+\beta)(\alpha+\beta) \leq  f^{-1}_\lambda(\alpha)\alpha+ \frac{r+1}{r} |\beta| \, | f^{-1}_{\lambda}(\alpha +\theta \beta)|.
\end{equation}

\end{lem}
\begin{proof} Consider the function $m(t) = f_{\lambda}^{-1}(t)t$. Then, $m$ is even, $m'(0) = 0$ and 
\[
m'(t)    =   \displaystyle f_{\lambda}^{-1}(t)+\frac{t f_{\lambda}^{-1}(t)}{\lambda r \left[f_{\lambda}^{-1}(t)\right]^{r}+p\left[f_{\lambda}^{-1}(t)\right]^{p}}, \ \ \text{for $t>0$}.
\]
Hence,
\[
0 < m'(t) < f_{\lambda}^{-1}(t)+\frac{t f_{\lambda}^{-1}(t)}{\lambda r \left[f_{\lambda}^{-1}(t)\right]^{r}+r\left[f_{\lambda}^{-1}(t)\right]^{p}} = \frac{r+1}{r} f_{\lambda}^{-1}(t) \ \ \forall \ t>0,
\]
which implies that $
|m'(t)| \leq \frac{r+1}{r} |f_{\lambda}^{-1}(t)|$ for all $t \in \R.$ By the mean value theorem, there exists $\theta \in (0,1)$ such that
\[
0 \leq f^{-1}_\lambda(\alpha+\beta)(\alpha+\beta)= f^{-1}_\lambda(\alpha)\alpha+m'(\alpha + \theta \beta) \beta \leq  f^{-1}_\lambda(\alpha)\alpha+ \frac{r+1}{r} |\beta| \, | f^{-1}_{\lambda}(\alpha +\theta \beta)|.\qedhere\]

\end{proof}
\begin{lem}\label{lemac.a.}
\[
\displaystyle\lim_{t\to \infty} \frac{t^{\frac{p+1}{p}}-f_\lambda^{-1}(t)t}{t^\frac{r+1}{p}}=\frac{\lambda}{p}.
\]
In particular, given $0 < c < \frac{\lambda}{p}$, there exists $t_0>0$ such that
\begin{equation}\label{comportamentoinfinito}
0 \leq f_\lambda^{-1}(t)t\leq |t|^\frac{p+1}{p} -c\, |t|^{\frac{r+1}{p}}, \quad \forall \, |t|\geq t_0.
\end{equation}
\end{lem}

\begin{proof}

For $t> 0$, writing $t=f_\lambda(\tau)$,
\begin{multline}
\displaystyle\frac{t^{\frac{p+1}{p}}-f_\lambda^{-1}(t)t}{t^\frac{r+1}{p}}=\frac{(\lambda \tau^r +\tau^p)^\frac{p+1}{p}-\lambda \tau^{r+1}-\tau^{p+1}}{\left(\lambda \tau^r+\tau^p\right)^\frac{r+1}{p}}
=\frac{(\lambda \tau^{r-\frac{p(r+1)}{p+1}} +\tau^{p-\frac{p(r+1)}{p+1}})^\frac{p+1}{p}-\lambda-\tau^{p-r}}{\tau^\frac{(r-p)(r+1)}{p}\left(\lambda +\tau^{p-r}\right)^\frac{r+1}{p}}\\
=\frac{(\lambda+\tau^{p-r})[\tau^\frac{r-p}{p}(\lambda +\tau^{p-r})^\frac{1}{p}-1]}{\tau^\frac{(r-p)(r+1)}{p}\left(\lambda +\tau^{p-r}\right)^\frac{r+1}{p}}=\frac{\tau^\frac{r-p}{p}(\lambda +\tau^{p-r})^\frac{1}{p}-1}{\tau^\frac{(r-p)(r+1)}{p}\left(\lambda +\tau^{p-r}\right)^{\frac{r+1}{p}-1}}.
\end{multline}
Then, with $y=\tau^{r-p}$, $y\overset{t\to\infty}{\longrightarrow}0^+$, and it follows that 
\begin{equation}
\displaystyle\lim_{t\to \infty} \frac{t^{\frac{p+1}{p}}-f_\lambda^{-1}(t)t}{t^\frac{r+1}{p}}=\displaystyle\lim_{y\to0^+}\frac{y^{1/p}(\lambda + y^{-1})^{1/p}-1}{y^\frac{r+1}{p}(\lambda +y^{-1})^{\frac{r+1}{p}-1}}=\displaystyle\lim_{y\to0^+}\frac{(\lambda y + 1)^{1/p}-1}{y(\lambda y +1)^{\frac{r+1}{p}-1}},
\end{equation}
and applying the L'H\^opital rule,
\[
\displaystyle\lim_{t\to \infty} \frac{t^{\frac{p+1}{p}}-f_\lambda^{-1}(t)t}{t^\frac{r+1}{p}}=\lim_{y\to0^+}\frac{\frac{\lambda}{p}(\lambda y + 1)^{1/p-1}}{(\lambda y +1)^{\frac{r+1}{p}-1}+\left(\frac{r+1}{p}-1 \right)\lambda y(\lambda y +1)^{\frac{r+1}{p}-2}}=\frac{\lambda}{p}. \qedhere
\]
\end{proof}

\subsection{Upper bound for the Mountain Pass level}\label{sectionMP}

Let $(p,q)$ be on the critical hyperbola \eqref{pqHC} and $(\varphi, \psi)$ be a positive radial solution of the problem 
\begin{equation}\label{instanton}
-\Delta \varphi = |\psi|^{p-1}\psi, \ \ -\Delta \psi = |\varphi|^{q-1} \varphi, \ \ \text{in} \ \ \mathbb{R}^N,
\end{equation}
whose qualitative and quantitative properties are described in \cite{Hulshof-VanderVorst}. 

We recall that $(\varphi,\psi)$ has de following decay at infinity:
\begin{equation}
\begin{array}{lllll}
a) & \frac{2}{N-2}<p<\frac{N}{N-2}, & \displaystyle \lim_{t\rightarrow \infty} t^{p(N-2)-2}\varphi(t)=b & \text{and} & \displaystyle\lim_{t\rightarrow \infty} t^{N-2}\psi(t)=c, \\\vspace{2pt}

b) & p=\frac{N}{N-2}, & \displaystyle\lim_{t\rightarrow \infty} \dfrac{t^{N-2}}{\log{t}}\varphi(t)=b & \text{and} & \displaystyle\lim_{t\rightarrow \infty} t^{N-2}\psi(t)=c, \\ \vspace{2pt}

c) & \frac{N}{N-2}<p<\frac{N^2+2N-4}{N^2-4N+4}, & \displaystyle\lim_{t\rightarrow \infty} t^{N-2}\varphi(t)=b & \text{and} & \displaystyle\lim_{t\rightarrow \infty} t^{N-2}\psi(t)=c, \\\vspace{2pt}

d) & p=\frac{N^2+2N-4}{N^2-4N+4}, & \displaystyle\lim_{t\rightarrow \infty} t^{N-2}\varphi(t)=b & \text{and} &\displaystyle \lim_{t\rightarrow \infty} \dfrac{t^{N-2}}{\log{t}}\psi(t)=c, \\ \vspace{2pt}

e) & \frac{N^2+2N-4}{N^2-4N+4}<p, &\displaystyle \lim_{t\rightarrow \infty} t^{N-2}\varphi(t)=b & \text{and} &\displaystyle \lim_{t\rightarrow \infty} t^{q(N-2)-2}\psi(t)=c,
\end{array}
\end{equation}
where $b>0$ and $c>0$ are constants and $t=|x|$. Fix $a \in \Omega$. Let $\xi_a \in {\it C}_c^{\infty} (\R^N)$ be a function such that $0 \leq \xi_a(x)\leq 1$ for all $x \in \R^N$, $\xi_a \equiv 1$ in $B(a, \rho/2)$, $\xi_a \equiv 0$ in $B(a, \rho)^{c}$ and $B(a, \rho) \subset \subset \Omega$, $\rho >0$.

\begin{lem}\label{estimateFtv}
Suppose \eqref{rscondition} and let $U_{\delta,a}:=\delta^{\frac{-N}{q+1}}\xi_a(x)  \varphi\left(\frac{x-a}{\delta}\right)$, where $\varphi$ is defined by \eqref{instanton} and $V_{\delta,a}=|U_{\delta,a}|^{-1}_{q+1}U_{\delta,a}$. Then, for everery $t\in[m,\overline{m}]$, with $m>0$ and $m, \, \overline{m}$ independent of $\delta$, it holds:

\noindent $\bullet$ if $N=3$ with $p \leq 7/2$, or $N\geq4$ with $p\leq\frac{N+2}{N-2}$, then
\begin{equation}\label{termoLaplacianoA}
\int_\Omega f^{-1}_\lambda(t\Delta V_{\delta,a}(x)t\Delta V_{\delta,a}(x))dx<t^\frac{p+1}{p}S, \ \ \ \ \text{for $\delta>0$ suitably small}, 
\end{equation}
$\bullet$ if $N=3$ with $7/2<p<11$, or $N\geq4$ with $\frac{N+2}{N-2}<p\leq\frac{N^2+2N-4}{N^2-4N+4}$, then
\begin{equation}\label{termoLaplacianoB}
\!\!\!\int_\Omega f^{-1}_\lambda(t\Delta V_{\delta,a}(x)t\Delta V_{\delta,a}(x))dx\!<\!
\left\{ \begin{array}{lll}
 t^\frac{p+1}{p}S+c_1t^\frac{r+1}{r} \delta^{\frac{N(r+1)}{r(p+1)}}-c_2 t^\frac{r+1}{p}\delta^{\frac{N}{q+1}\frac{N}{N-2}}, &\text{if }\, r<\frac{2}{N-2},\\
 t^\frac{p+1}{p}S+c_1t^\frac{r+1}{r} \delta^{\frac{N(r+1)}{r(p+1)}}-c_2\delta^{\frac{N}{q+1}\frac{N}{N-2}}|\log(\delta)|, &\text{if }\, r=\frac{2}{N-2},\\
  t^\frac{p+1}{p}S+c_1t^\frac{r+1}{r} \delta^{\frac{N(r+1)}{r(p+1)}}-c_2 t^\frac{r+1}{p}\delta^{\frac{N(p-r)}{p+1}}, &\text{if }\, r>\frac{2}{N-2},
\end{array}
\right.
\end{equation}

$\bullet$ if $\frac{N^2+2N-4}{N^2-4N+4}<p$, then
\begin{equation}\label{termoLaplacianoC}
\!\!\!\int_\Omega f^{-1}_\lambda(t\Delta V_{\delta,a}(x)t\Delta V_{\delta,a}(x))dx\!<\!
\left\{ \begin{array}{lll}
 t^\frac{p+1}{p}S+c_1t^\frac{r+1}{r} \delta^{\frac{N(r+1)}{r(p+1)}}- c_2\lambda\delta^{\frac{Nq}{q+1}}|\log(\delta)|, & \! \text{if } r+1=\frac{p+1}{q+1},\\
 t^\frac{p+1}{p}S+c_1t^\frac{r+1}{r} \delta^{\frac{N(r+1)}{r(p+1)}}-c_2\lambda\delta^{\frac{Nq}{q+1}}, & \! \text{if } r+1<\frac{p+1}{q+1},\\
  t^\frac{p+1}{p}S+c_1t^\frac{r+1}{r} \delta^{\frac{N(r+1)}{r(p+1)}}-c_2\lambda\delta^{\frac{N(p-r)}{p+1}}, &  \! \text{if }  r+1>\frac{p+1}{q+1}.
\end{array}
\right.
\end{equation}
 \end{lem}
\begin{proof}
First, observe that
\[
\Delta V_\delta = \gamma_{\delta,a}(x) + \sigma_{\delta,a}(x), 
\]
where
 $$\gamma_{\delta,a}(x):= |U_\delta|^{-1}_{q+1}\delta^{\frac{-N}{q+1}}\delta^{-2} \Delta \varphi\left(\frac{x-a}{\delta}\right)$$
and
$$\sigma_{\delta,a}(x):=|U_\delta|^{-1}_{q+1}\delta^{\frac{-N}{q+1}} \left(2\delta^{-1} \nabla \xi_a(x)  \nabla \varphi\left(\frac{x-a}{\delta}\right)+\varphi\left(\frac{x-a}{\delta}\right)\Delta \xi_a(x)\right).
$$
So,
\begin{multline}
\displaystyle \int_\Omega f^{-1}_\lambda(t\Delta V_{\delta,a}(x))t\Delta V_{\delta,a}(x)dx
=\\\displaystyle  \int_{B(a,\rho/2)} f^{-1}_\lambda(t\xi_a(x)\gamma_{\delta,a}(x)+t\sigma_{\delta,a}(x))(t\xi_a(x)\gamma_{\delta,a}(x)
+t\sigma_{\delta,a}(x))dx\\
+\displaystyle\int_{\Omega\backslash B(a,\rho/2)}f^{-1}_\lambda(t\xi_a(x)\gamma_{\delta,a}(x)+t\sigma_{\delta,a}(x))(t\xi_a(x)\gamma_{\delta,a}(x)+t\sigma_{\delta,a}(x))dx,
\end{multline}
and since $supp\left(\Delta V_{\delta,a}(x)\right) \subset \overline{B}(a,\rho)$ and $supp( t\sigma_{\delta,a}(x))\subset \overline{R(a,\rho/2,\rho)}$, where $R(a, \rho/2, \rho):=B(a,\rho)\setminus \overline{B}(a,\rho/2)$, one has

\begin{multline}\label{eq:tudo}
\displaystyle \int_\Omega f^{-1}_\lambda(t\Delta V_{\delta,a}(x))t\Delta V_{\delta,a}(x)dx
=\displaystyle  \int_{B(a,\rho/2)} f^{-1}_\lambda(t\gamma_{\delta,a}(x))t\gamma_{\delta,a}(x)dx\\
+\displaystyle\int_{R(a,\rho/2,\rho)}f^{-1}_\lambda(t\xi_a(x)\gamma_{\delta,a}(x)+t\sigma_{\delta,a}(x))(t\xi_a(x)\gamma_{\delta,a}(x)+t\sigma_{\delta,a}(x))dx.
\end{multline}

Now we split the estimate in two steps, which correspond to the principal part
\begin{equation}\label{principalpart}
h_{\delta,a}:= \int_{B(a,\rho/2)}  f^{-1}_\lambda\left(t \gamma_{\delta,a}(x)\right)t \gamma_{\delta,a}(x),
\end{equation}
and to the residual part
\begin{equation}\label{residualpart}
j_{\delta,a}:= \displaystyle\int_{R(a,\rho/2,\rho)}f^{-1}_\lambda(t\xi_a(x)\gamma_{\delta,a}(x)+t\sigma_{\delta,a}(x))(t\xi_a(x)\gamma_{\delta,a}(x)+t\sigma_{\delta,a}(x))dx.
\end{equation}

\smallskip
\noindent \textbf{Step 1.} Estimate of \eqref{principalpart}

Using \eqref{ineqinv} and the asymptotic behavior of $\Delta \varphi$ as in \cite[Theorem 2]{Hulshof-VanderVorst} and \cite[Lemma 6.2]{ederson}, one gets

\begin{equation}\begin{array}{ccc}\label{eqparah}
 \displaystyle h_{\delta,a}=\int_{\R^N} \left|t \gamma_{\delta,a}(x)\right|^\frac{p+1}{p}dx + \int_{B(a,\rho/2)} f^{-1}_\lambda\left(t \gamma_{\delta,a}(x)\right)t \gamma_{\delta,a}(x)dx-\int_{\R^N} \left|t \gamma_{\delta,a}(x)\right|^\frac{p+1}{p}dx\\
\displaystyle = t^\frac{p+1}{p}S+ \int_{B(a,\rho/2)} f^{-1}_\lambda\left(t \gamma_{\delta,a}(x)\right)t \gamma_{\delta,a}(x)-|t\gamma_{\delta,a}(x)|^\frac{p+1}{p}dx-\int_{\R^N\setminus B(a,\rho/2)} \left|t \gamma_{\delta,a}(x)\right|^\frac{p+1}{p}dx.
 \end{array}
\end{equation}
The behavior of the last term is already known by \cite{ederson}, namely
\begin{equation}\label{ineqRN}
-\int_{\R^N\setminus B(a,\rho/2)} \left|t \gamma_{\delta,a}(x)\right|^\frac{p+1}{p}dx\leq \left\{
\begin{array}{lll}
-t^\frac{p+1}{p}C\delta^{\frac{N(p+1)}{q+1}},&\text{ if }q>\frac{N}{N-2},\\
-t^\frac{p+1}{p}C|\log(\delta)|^{p+1}\delta^{\frac{N(p+1)}{q+1}},&\text{ if }q=\frac{N}{N-2},\\
-t^\frac{p+1}{p}C\delta^{qN},&\text{ if }q<\frac{N}{N-2}.
\end{array}\right.
\end{equation}
Next, we estimate
\begin{equation}\label{eqdefi}
i_{\delta,a}:=\int_{B(a,\rho/2)} f^{-1}_\lambda\left(t \gamma_{\delta,a}(x)\right)t \gamma_{\delta,a}(x)- |t\gamma_{\delta,a}(x)|^\frac{p+1}{p} dx.
\end{equation}
We consider three parts of the ball $B(a,\rho/2)$, namely the ball $B(a,\delta)$ and the rings $R(a,\delta,\delta^M)$ and $R(a,\delta^M,\rho/2)$, where the number $M<1$ will be defined ahead. This splitting involving rings is key argument to capture the contribution of the term $\lambda |u|^{r-1}r$ to downsize the Mountain Pass level.
\medskip

\noindent \textbf{Step 1.1.} By the behavior of $\Delta \varphi(x)$ it is known that there exists (for $\delta$ sufficiently small) $c>0$ such that if $|x-a|<\delta$ then $c<t|\Delta \varphi(\frac{x-a}{\delta})|$, so \eqref{comportamentoinfinito} and $\frac{N+2(q+1)}{q+1} = \frac{pN}{p+1}$, can be used to infer that
\begin{multline}\label{eqtodopq}
\displaystyle\int_{B(a,\delta)} f^{-1}_\lambda\left(t \gamma_{\delta,a}(x)\right)t \gamma_{\delta,a}(x)-|t\gamma_{\delta,a}(x)|^\frac{p+1}{p}dx\displaystyle\leq -c\lambda \int_{B(a,\delta)}|t\gamma_{\delta,a}(x)|^\frac{r+1}{p}dx\\
\displaystyle \displaystyle\leq -c'\lambda \int_{B(a,\delta)}\delta^{-N\frac{r+1}{p+1}}dx=-C\lambda \delta^{N\frac{p-r}{p+1}}.
\end{multline}
\textbf{Step 1.2.} Now focus the attention on the $R(a,\delta,\delta^M)$-term. In this ring, since $1 < \frac{|x-a|}{\delta} < \delta^{M-1}$, with $M<1$ to be defined, the asymptotic decay of $\gamma_{\delta,a}(x)$ present in \cite[Lemma 6.2]{ederson} can be used, and three cases have to be analized.

\noindent {\bf Case 1:} $q>\frac{N}{N-2}$.\\
In this case it is known that
\[ \delta^\frac{-pN}{p+1}\left|\Delta \varphi\left(\frac{x-a}{\delta}\right)\right|\geq c |x-a|^{-p(N-2)}\delta^{p(N-2)-\frac{pN}{p+1}},\]
the last term is grater than a constant if $|x-a|\leq\delta^M$ with
 \[(1-M)p(N-2)-\frac{pN}{p+1}=0\iff M=\frac{p(N-2)-2}{(N-2)(p+1)}=\frac{N}{N-2}\frac{1}{q+1},\]
observing that $0<M<1$, and $\delta^M<\rho/2$ since $\delta \to 0$. Applying Lemma \ref{lemac.a.}, it follows that
\begin{multline}
\int_{R(a,\delta,\delta^M)} f^{-1}_\lambda\left(t \gamma_{\delta,a}(x)\right)t \gamma_{\delta,a}(x)-|t\gamma_{\delta,a}(x)|^\frac{p+1}{p}dx\leq -c\lambda\int_{R(a,\delta,\delta^M)} \left|t \gamma_{\delta,a}(x)\right|^\frac{r+1}{p}dx\\
  \leq\displaystyle-c\lambda \int_{R(a,\delta,\delta^M)} |x-a|^{-(r+1)(N-2)}\delta^{N\frac{r+1}{q+1}}dx
  \displaystyle=  -c\lambda\delta^{N\frac{r+1}{q+1}}\int_\delta^{\delta^M} y^{1-r(N-2)}dy,
  \end{multline}
  and hence
\[
 \displaystyle  \int_{R(a,\delta,\delta^M)} f^{-1}_\lambda\left(t \gamma_{\delta,a}(x)\right)t \gamma_{\delta,a}(x)-|t\gamma_{\delta,a}(x)|^\frac{p+1}{p}dx\leq
-\left\{\begin{array}{lll}
   c\lambda\delta^{\frac{N}{q+1}\frac{N}{N-2}}|\log(\delta)|, & \text{if } r=\frac{2}{N-2},\\
 \displaystyle c\lambda \frac{\delta^{\frac{N}{q+1}\frac{N}{N-2}}-\delta^{N\frac{p-r}{p+1}}}{2-r(N-2)}, & \text{if } r\neq\frac{2}{N-2}.
 \end{array}\right.
\]
 {\bf Case 2:} $q=\frac{N}{N-2}$.\\
In this case it is known that
 \[ \delta^\frac{-pN}{p+1}\left|\Delta \varphi\left(\frac{x-a}{\delta}\right)\right|\geq c \left(\log\left(\frac{|x-a|}{\delta}\right)+1\right)^{p} |x-a|^{-p(N-2)}\delta^{p(N-2)-\frac{pN}{p+1}},\]
so one can use $M=\frac{N}{N-2}\frac{1}{q+1}$ and proceed as in Case 1, to obtain the same estimate, which could be even better.

\noindent {\bf Case 3:} $q<\frac{N}{N-2}$.\\
In this case it is known that
 \[ \delta^\frac{-pN}{p+1}\left|\Delta \varphi\left(\frac{x-a}{\delta}\right)\right|\geq c |x-a|^{-\frac{p(q+1)N}{p+1}}\delta^{\frac{pqN}{p+1}},\]
 and this is grater than a constant if $|x-a|\leq\delta^M$ with
 \[ \frac{Mp(q+1)N}{p+1}=\frac{pqN}{p+1} \iff M=\frac{q}{q+1}.\]
So Lemma \ref{lemac.a.} can be applied, leading that

 \begin{multline}
\int_{R(a,\delta,\delta^M)} f^{-1}_\lambda\left(t \gamma_{\delta,a}(x)\right)t \gamma_{\delta,a}(x)-|t\gamma_{\delta,a}(x)|^\frac{p+1}{p}dx\leq -c\lambda\int_{R(a,\delta,\delta^M)} \left|t \gamma_{\delta,a}(x)\right|^\frac{r+1}{p}dx\\ 
 \leq \displaystyle-c\lambda\int_{R(a,\delta,\delta^M)} |x-a|^{-\frac{(r+1)(q+1)}{p+1}N}\delta^{\frac{r+1}{p+1}qN}dx
  \displaystyle=-c\lambda \delta^{\frac{r+1}{p+1}qN}\int_\delta^{\delta^M} y^{N-1-\frac{(q+1)(r+1)}{p+1}N}dy
 \end{multline}
 and so
 \[
\displaystyle \int_{R(a,\delta,\delta^M)} f^{-1}_\lambda\left(t \gamma_{\delta,a}(x)\right)t \gamma_{\delta,a}(x)-|t\gamma_{\delta,a}(x)|^\frac{p+1}{p}dx\leq
   -\left\{\begin{array}{lll}
   c\lambda\delta^{\frac{Nq}{q+1}}|\log(\delta)|, & \text{if } r+1=\frac{p+1}{q+1},\\
 \displaystyle c\lambda \frac{\delta^{\frac{Nq}{q+1}}-\delta^{N\frac{p-r}{p+1}}}{N-N\frac{(q+1)(r+1)}{p+1}}, & \text{if } r+1\neq\frac{p+1}{q+1}.
 \end{array}\right.
\]

\noindent
{\bf Step 1.3.} Finally we estimate the $R(a,\delta^M,\rho/2)$-term. In this ring, $\left| \frac{x-a}{\delta} \right| >1$, and the asymptotic behavior of $\gamma_{\delta,a}(x)$ present in \cite[Lemma 6.2]{ederson} can be used one more time, but in this case $\gamma_{\delta,a}$ becomes small, and then $ f^{-1}_\lambda\left(t \gamma_{\delta,a}(x)\right)t \gamma_{\delta,a}(x)-|t\gamma_{\delta,a}(x)|^\frac{p+1}{p}\leq -c_\lambda|t\gamma_{\delta,a}(x)|^\frac{p+1}{p}$, and one can proceed as in \eqref{ineqRN}. 

At this point, from Steps 1.1, 1.2, and 1.3, we can write the estimates for $i_{\delta, a}$ defined in \eqref{eqdefi}. But before doing this, note that in all the three cases of Step 1.2, the term $\delta^\frac{N(p-r)}{p+1}$ (dominant term in Step 1.1) appear, so it does not need to be repeated.

\begin{equation}\label{tudodei}
i_{\delta,a}\leq\left\{
\begin{array}{rl}
-c\lambda\delta^{\frac{N}{q+1}\frac{N}{N-2}}|\log(\delta)|-C\delta^{\frac{N(p+1)}{q+1}}, &\text{if }q>\frac{N}{N-2},\, r=\frac{2}{N-2},\\

- \displaystyle c\lambda \frac{\delta^{\frac{N}{q+1}\frac{N}{N-2}}-\delta^{N\frac{p-r}{p+1}}}{2-r(N-2)}-C\delta^{\frac{N(p+1)}{q+1}}, & \text{if }q>\frac{N}{N-2}, \, r\neq\frac{2}{N-2},\\

-c\lambda\delta^{\frac{N}{q+1}\frac{N}{N-2}}|\log(\delta)|-C|\log(\delta)|^{p+1}\delta^{\frac{N(p+1)}{q+1}}, & \text{if }q=\frac{N}{N-2},\,  r=\frac{2}{N-2},\\

- c\lambda \frac{\delta^{\frac{N}{q+1}\frac{N}{N-2}}-\delta^{N\frac{p-r}{p+1}}}{2-r(N-2)}-C|\log(\delta)|^{p+1}\delta^{\frac{N(p+1)}{q+1}}, & \text{if }q=\frac{N}{N-2}, \, r\neq\frac{2}{N-2},\\

-c\lambda\delta^{\frac{Nq}{q+1}}|\log(\delta)|- C\delta^{qN}, & \text{if }q<\frac{N}{N-2}, \, r+1=\frac{p+1}{q+1},\\

 -\displaystyle c\lambda \frac{\delta^{\frac{Nq}{q+1}}-\delta^{N\frac{p-r}{p+1}}}{N-N\frac{(q+1)(r+1)}{p+1}}- C\delta^{qN}, & \text{if }q<\frac{N}{N-2},\, \,r+1\neq\frac{p+1}{q+1}.

\end{array}\right.
\end{equation}

Now we summarize all the calculation made in Step 1.  To estimate $h_{\delta, a}$, from \eqref{eqparah}, we must deal with \eqref{ineqRN} and \eqref{tudodei}. Observe that the majorante in \eqref{ineqRN} also appears in the second terms in \eqref{tudodei}. Then, the estimate for \eqref{eqparah} follows from the comparison among the powers of $\delta$ in \eqref{tudodei}.

For $q>\frac{N}{N-2}$, the terms to be compared are

\begin{equation}\label{comp.q>}
-C\delta^{\frac{N(p+1)}{q+1}} \text{ and }-\left\{\begin{array}{lll}
   c\lambda\delta^{\frac{N}{q+1}\frac{N}{N-2}}|\log(\delta)|, & \text{if } r=\frac{2}{N-2},\\
 \displaystyle c\lambda \frac{\delta^{\frac{N}{q+1}\frac{N}{N-2}}-\delta^{N\frac{p-r}{p+1}}}{2-r(N-2)}, & \text{if } r\neq\frac{2}{N-2},
 \end{array}\right.
 \end{equation}
 and the first part is always weaker. Of course $N\frac{p-r}{p+1}>\frac{N}{q+1}\frac{N}{N-2}$ if $r<\frac{2}{N-2}$, and $\frac{N(p+1)}{q+1}>\frac{N}{q+1}\frac{N}{N-2}$ (since $p>\frac{2}{N-2}$ always), so in this case the dominant term is $-c\lambda\delta^{\frac{N}{q+1}\frac{N}{N-2}}$. The same analysis shows that if $r=\frac{2}{N-2}$ the dominant term is $ c\lambda\delta^{\frac{N}{q+1}\frac{N}{N-2}}|\log(\delta)|$. If $r>\frac{2}{N-2}$, then $N\frac{p-r}{p+1}<\frac{N}{q+1}\frac{N}{N-2}$ and the dominant term is $-c \delta^{N\frac{p-r}{p+1}}$.  

When $q=\frac{N}{N-2}$ the analysis done before gives that the dominant term is
\[-\left\{\begin{array}{lll}
   c\lambda\delta^{\frac{N}{q+1}\frac{N}{N-2}}|\log(\delta)|, & \text{if } r=\frac{2}{N-2},\\
 \displaystyle c\lambda \delta^{\frac{N}{q+1}\frac{N}{N-2}}, & \text{if } r<\frac{2}{N-2},\\
 -t^\frac{p+1}{p}C|\log(\delta)|^{p+1}\delta^{\frac{N(p+1)}{q+1}}-C\lambda \delta^{N\frac{p-r}{p+1}}&\text{if } r>\frac{2}{N-2}.
 \end{array}\right.\]
Finaly, if $q<\frac{N}{N-2}$ the terms that we have to compare are
\[-t^\frac{p+1}{p}C\delta^{qN} \text{ and }-\left\{\begin{array}{lll}
   c\lambda\delta^{\frac{Nq}{q+1}}|\log(\delta)|, & \text{if } r+1=\frac{p+1}{q+1},\\
 \displaystyle c\lambda \frac{\delta^{\frac{Nq}{q+1}}-\delta^{N\frac{p-r}{p+1}}}{N-N\frac{(q+1)(r+1)}{p+1}}, & \text{if } r+1\neq\frac{p+1}{q+1}.
 \end{array}\right.\]
If $r+1<\frac{p+1}{q+1}$, it is easy to see that $\frac{Nq}{q+1}<N\frac{p-r}{p+1}$ and surely $\frac{Nq}{q+1}<Nq,$ so $-c\lambda\delta^{\frac{Nq}{q+1}}$ is the dominant term. The same computation ensures that the dominant term is $-c\lambda\delta^{\frac{Nq}{q+1}}|\log(\delta)|,$  if  $r+1=\frac{p+1}{q+1}$. Finally, if $r+1>\frac{p+1}{q+1}$ an analogous computation show that the term $-C\lambda \delta^{N\frac{p-r}{p+1}}$ is the dominant. Then, putting it all together, 
\begin{equation}\label{feitoh}
h_{\delta,a}\leq t^{\frac{p+1}{p}}S - \left\{\begin{array}{rl}
c\lambda\delta^{\frac{N}{q+1}\frac{N}{N-2}}|\log(\delta)|, &\text{if }q\geq\frac{N}{N-2},\, r=\frac{2}{N-2},\\

 c\lambda \delta^{\frac{N}{q+1}\frac{N}{N-2}}, & \text{if }q\geq\frac{N}{N-2}, \, r<\frac{2}{N-2},\\
 c\lambda\delta^{\frac{N(p-r)}{p+1}}, & \text{if }q\geq\frac{N}{N-2}, \,  r>\frac{2}{N-2},\\

 \lambda\delta^{\frac{Nq}{q+1}}|\log(\delta)|, & \text{if }q<\frac{N}{N-2}, \, r+1=\frac{p+1}{q+1},\\
   
 c\lambda\delta^{\frac{Nq}{q+1}}, & \text{if }q<\frac{N}{N-2},\, \,r+1<\frac{p+1}{q+1},\\

 c\lambda\delta^{\frac{N(p-r)}{p+1}}, &  \text{if }q<\frac{N}{N-2},\, \,r+1>\frac{p+1}{q+1}.
\end{array}\right.
\end{equation}

\noindent{\bf Step 2} Estimate of the residual part \eqref{residualpart}.

Here $\left| \frac{x-a}{\delta} \right| \geq  \frac{\rho}{2\delta} \to \infty$, uniformly with respect to $x \in R(a,\rho/2,\rho)$, as $\delta \to 0$. Then the asymptotic behavior of $\gamma_{\delta,a }$ and $\sigma_{\delta,a}$ present in \cite[Lemma 6.2]{ederson} reads
\begin{equation}
\sigma_{\delta,a}(x) \leq\left\{
\begin{array}{ll}
 c|U_\delta|^{-1}_{q+1} \delta^{\frac{N}{p+1}}  (|x-a|^{-N+1}+|x-a|^{-N+2}),&\text{ if } p>\frac{N}{N-2},\\
 c|U_\delta|^{-1}_{q+1} \delta^{\frac{N}{p+1}} |\log(\delta)| (|\log|x-a|+1|)(|x-a|^{-N+1}(1+|x-a|)),&\text{ if } p=\frac{N}{N-2},\\
 c|U_\delta|^{-1}_{q+1} \delta^{\frac{pN}{q+1}}(|x-a|^{-p(N-2)+2}+|x-a|^{-p(N-2)+1}),& \text{ if } p<\frac{N}{N-2},
\end{array}\right.
\end{equation}
and,
\begin{equation}
\gamma_{\delta,a} \leq\left\{
\begin{array}{ll}
 c |x-a|^{-p(N-2)}\delta^{\frac{pN}{q+1}},&\text{ if } q>\frac{N}{N-2}\\
 c (\left|\log\left(|x-a|\right)\right|^{p+1}+1)^\frac{p}{p+1} |x-a|^{-p(N-2)}|\log\delta|^p\delta^{\frac{pN}{q+1}},&\text{ if } q=\frac{N}{N-2}\\
 c |x-a|^{\frac{p(q+1)N}{p+1}}\delta^{\frac{pqN}{p+1}}, &\text{ if } q<\frac{N}{N-2}.
\end{array}\right.
\end{equation}

From this, it follows that $|j_{\delta,a}|$ is bounded from above by
\begin{equation}
 |R(a, \rho/2, \rho)|\left\{\begin{array}{lll}
 f^{-1}_\lambda\left(ct\left( \delta^{\frac{pN}{q+1}}+ \delta^\frac{pN}{q+1}\right)\right)ct\left( \delta^{\frac{pN}{q+1}}+ \delta^{\frac{pN}{q+1}}\right),&\text{if }p<\frac{N}{N-2}, \\

 f^{-1}_\lambda\left(ct\left( \delta^{\frac{pN}{q+1}}+ \delta^{\frac{N}{p+1}} |\log(\delta)|\right)\right)ct\left( \delta^{\frac{pN}{q+1}}+ \delta^{\frac{N}{p+1}} |\log(\delta)|\right),& \text{if }p=\frac{N}{N-2},\\
 f^{-1}_\lambda\left(ct\left( \delta^{\frac{pN}{q+1}}+ \delta^{\frac{N}{p+1}}\right)\right)ct\left( \delta^{\frac{pN}{q+1}}+ \delta^{\frac{N}{p+1}}\right),& \text{if }p,q>\frac{N}{N-2}, \\
 f^{-1}_\lambda\left(ct\left( |\log\delta|^p\delta^{\frac{pN}{q+1}}+ \delta^{\frac{N}{p+1}}\right)\right)ct\left( |\log\delta|^p\delta^{\frac{pN}{q+1}}+ \delta^{\frac{N}{p+1}}\right),& \text{if }q=\frac{N}{N-2}, \\
 f^{-1}_\lambda\left(ct\left( \delta^{\frac{pqN}{p+1}}+ \delta^{\frac{N}{p+1}}\right)\right)ct\left( \delta^{\frac{pqN}{p+1}}+ \delta^{\frac{N}{p+1}}\right),& \text{if }q<\frac{N}{N-2},
\end{array}\right.
\end{equation}
\begin{equation}
\leq |R(a, \rho/2, \rho)|\left\{\begin{array}{lll}
 f^{-1}_\lambda\left(tc\delta^\frac{pN}{q+1}\right)tc\delta^{\frac{pN}{q+1}},&\text{ if }p<\frac{N}{N-2}, \\
 f^{-1}_\lambda\left(tc\delta^{\frac{N}{p+1}} |\log(\delta)|\right)tc\delta^{\frac{N}{p+1}} |\log(\delta)|,& \text{ if }p=\frac{N}{N-2},\\
 f^{-1}_\lambda\left(tc\delta^{\frac{N}{p+1}}\right)tc\delta^{\frac{N}{p+1}},& \text{ if }p,q>\frac{N}{N-2}, \\
 f^{-1}_\lambda\left(tc\delta^{\frac{N}{p+1}}\right)tc\delta^{\frac{N}{p+1}},& \text{ if }q=\frac{N}{N-2}, \\
 f^{-1}_\lambda\left(tc\delta^{\frac{N}{p+1}}\right)tc\delta^{\frac{N}{p+1}},& \text{ if }q<\frac{N}{N-2},
\end{array}\right.
\end{equation}
\begin{equation}
= |R(a, \rho/2, \rho)|\left\{\begin{array}{rll}
 f^{-1}_\lambda\left(tc\delta^\frac{pN}{q+1}\right)tc\delta^{\frac{pN}{q+1}},&\text{ if }p<\frac{N}{N-2}, \\
 f^{-1}_\lambda\left(tc\delta^{\frac{N}{p+1}} |\log(\delta)|\right)tc\delta^{\frac{N}{p+1}} |\log(\delta)|,& \text{ if }p=\frac{N}{N-2},\\
 f^{-1}_\lambda\left(tc\delta^{\frac{N}{p+1}}\right)tc\delta^{\frac{N}{p+1}},& \text{ if }p>\frac{N}{N-2},
\end{array}\right.
\end{equation}
and using the asymptotic behavior of $f_\lambda^{-1}$, one concludes that
\begin{equation}\label{feitoj}
|j_{\delta,a}|\leq\left\{\begin{array}{rll}
c_\lambda\left(t\delta^\frac{pN}{q+1}\right)^\frac{r+1}{r},&\text{ if }p<\frac{N}{N-2},\\

 c_\lambda\left(t\delta^{\frac{N}{p+1}} |\log(\delta)|\right)^\frac{r+1}{r},& \text{ if }p=\frac{N}{N-2},\\

 c_\lambda\left(t\delta^{\frac{N}{p+1}}\right)^\frac{r+1}{r},& \text{ if }p>\frac{N}{N-2}.
\end{array}\right.
\end{equation}

\noindent \textbf{Step 3:}  Comparison of the residual parts in \eqref{feitoh} and \eqref{feitoj}. 

\noindent \textbf{Step 3.1: $p<\frac{N}{N-2}$.}  First, observe that this implies $q> \frac{N}{N-2}$. To obtain \eqref{termoLaplacianoA}, from the comparison \eqref{comp.q>} to obtain \eqref{feitoh}, it is enough to verify that $\frac{N(p+1)}{q+1}<\frac{pN}{q+1}\frac{r+1}{r}$, which is equivalent to 
$p+1<p\frac{r+1}{r}$, that is $r<p$, which is always the case. Hence, the  Lemma is proved in this case.

\noindent \textbf{Step 3.2: $p=\frac{N}{N-2}$.}  Again this implies $q> \frac{N}{N-2}$, and the procedure to obtain \eqref{termoLaplacianoA} is identical to that from Step 3.1.

\noindent \textbf{Step 3.3: $\frac{N}{N-2}<p\leq\frac{N+2}{N-2}$.}

If $r<\frac{2}{N-2}$, to obtain \eqref{termoLaplacianoA}, from \eqref{feitoh} and \eqref{feitoj}, one must decide when $\frac{N^2}{(q+1)(N-2)}< {\frac{N}{p+1}\frac{r+1}{r}}$, that is\\
\[\frac{N(p+1)}{(q+1)(N-2)}<\frac{r+1}{r}\iff \frac{p(N-2)-2}{N-2}-1<\frac{1}{r}\iff r<\frac{1}{p-\frac{N}{N-2}},\]
and this is always true for $N\geq 4$. If $N=3$, these conditions read
\[
3 < p \leq 5, \quad r < 2, \quad \text{and} \quad r < \frac{1}{p-3}
\]
which are satisfied with the extra condition $p \leq 7/2$.

If $r=\frac{2}{N-2}$, then $\frac{r+1}{r} = \frac{N}{2}$, to obtain \eqref{termoLaplacianoA}, from \eqref{feitoh} and \eqref{feitoj}, one must decide when \[\frac{p+1}{q+1}\leq\frac{N-2}{2}\iff \frac{p(N-2)-2}{N}\leq\frac{N-2}{2} \iff p\leq \frac{N}{2} + \frac{2}{N-2},\]
and (remember $p\leq \frac{N+2}{N-2}$) this is always true for $N \geq 4$, and with $N=3$ these conditions read $p \leq 7/2$.

If $r>\frac{2}{N-2}$, to obtain \eqref{termoLaplacianoA}, from \eqref{feitoh} and \eqref{feitoj}, one must decide when $\frac{N(p-r)}{p+1} <\frac{N}{p+1}\frac{r+1}{r}$, that is
\begin{equation}\label{polinomio} p-r<\frac{r+1}{r}\iff 0<r^2+(1-p)r+1.\end{equation}

Remember that $p\leq \frac{N+2}{N-2}$. Then, for $N > 4$, \eqref{polinomio} is true because such second order polynomial has no real roots. For $N=4$ and $p < \frac{N+2}{N-2}$, again \eqref{polinomio} has no real root and \eqref{polinomio} is verified. For $N=4$ and $p= \frac{N+2}{N-2}$, such polynomial has $1$ as real root and $r > 1 = 2/(N-2)$, hence  \eqref{polinomio} is verified. With $N=3$, the largest real root of such polynomial is less or equal to $2$ for $p \leq 7/2$, hence \eqref{polinomio} is verified because $r> 2 = 2/(N-2)$.  This finishes the proof of the lemma. \end{proof}

\begin{obs}
The estimate of $\int_\Omega f^{-1}_\lambda(t\Delta V_{\delta,a}(x))t\Delta V_{\delta,a}(x)dx$ from Lemma \ref{estimateFtv} deserves some comments.  
\end{obs}
When evaluating $\int_\Omega f^{-1}_\lambda(t\Delta V_{\delta,a}(x))t\Delta V_{\delta,a}(x)dx$, the leading term comes from \linebreak $\int_{B(0,\rho/2)} f^{-1}_\lambda\left(t\gamma_{\delta,a}(x)\right)t\gamma_{\delta,a}(x)dx$, which carries  by itself a (negative) remainder that has to be compared with the residual term $ \int_{R(0,\rho/2,\rho)} f_\lambda^{-1}(t\Delta V_{\delta,a}(x))t\Delta V_{\delta,a}(x) dx$. With $N\geq 4$ or $N=3$ and $p\leq7/2$ the remainder $\int_{R(0,\rho/2,\rho)}   f_\lambda^{-1}(t\Delta V_{\delta,a}(x))t\Delta V_{\delta,a}(x) dx$ is smaller than the negative part of the remainder term in $\int_{B(0,\rho/2)} f^{-1}_\lambda(t\Delta V_{\delta,a}(x))t\Delta V_{\delta,a}(x)dx$, which brings down the functional when comparing it to the problem without the perturbation $\lambda u^r$, and this plays an important role in the results in this paper. At this point, it is important to compare the estimates \eqref{termoLaplacianoA}, \eqref{termoLaplacianoB} and \eqref{termoLaplacianoC} with \cite[Eq. (6.4)]{ederson}.

\medskip We are almost prepared for the proof of Proposition \ref{prop-mplevel}. Going on this direction, observe that if \eqref{pqHC} and \eqref{rscondition} are satisfied, then
\begin{equation}
\lim_{t\to\infty}I_F(tV_{\delta,a})=-\infty
\end{equation}
and the $\max_{t\geq0}I(tV_{\delta,a})$ is achieved at some $t_\delta>0$, thus,
\begin{equation}\label{I'tv}
0=I_F'(t_\delta V_{\delta,a})=\displaystyle\int_\Omega f_\lambda^{-1}(t_\delta \Delta V_{\delta,a})\Delta V_{\delta,a}\, dx-t_\delta^s|V_{\delta,a}|^{s+1}_{s+1}-t_\delta^q,
\end{equation}
from where we infer that
\begin{equation}\label{tdeltaq+1}
t_\delta^{q+1}=\displaystyle\int_\Omega f_\lambda^{-1}(t_\delta \Delta V_{\delta,a})t_\delta\Delta V_{\delta,a}\, dx-t_\delta^{s+1}|V_{\delta,a}|^{s+1}_{s+1}.
\end{equation}
 
\begin{lem}\label{tdeltaltd}
Suppose \eqref{pqHC} and \eqref{rscondition}. Then $t_\delta$, as $\delta\rightarrow0$, is bounded form below and above. \end{lem}
\begin{proof}
Suppose by contradiction  that $t_\delta \overset{\delta\to 0}{\longrightarrow}0$. Define $A_\delta:=\{x\in \Omega;|t_\delta \Delta V_{\delta,a}(x)|<1+\lambda \}$ and $B_\delta=\Omega\backslash A_\delta$, small $\delta$ give us
\begin{multline}
Ct_\delta^{1/r}  \leq  \displaystyle \frac{t_\delta^{1/p}}{(1+\lambda)^{1/p}} \int_{B_\delta}  |\Delta V_{\delta,a} (x)|^\frac{p+1}{p}dx + \frac{t_\delta^{1/r}}{(1+\lambda)^{1/r}} \int_{A_\delta} |\Delta V_{\delta,a} (x)|^\frac{r+1}{r}dx\\\displaystyle \leq  \int_\Omega  f_\lambda^{-1}(t_\delta \Delta V_{\delta,a} (x))\Delta V_{\delta,a} (x)dx=t_\delta^s \,o(\delta)+t_\delta^q\, O(1).
\end{multline}
which is a contradiction by the fact that $1/r\leq s$ and $1/r<q$, so $t_\delta\nrightarrow0$. Observe that in the case $rs=1$ the $o(\delta)$ produces the contradiction.\\

Now observe that  by \eqref{I'tv} and the estimates present in \cite{EdJe-djairo}
\begin{equation}
t_\delta^{q} \leq t_\delta^\frac{1}{p} S+t_\delta^\frac{1}{p}o(\delta) \Longrightarrow t_\delta^\frac{pq-1}{p} \leq  S+o(\delta)
\end{equation}
that is, $t_\delta\leq k<\infty$ for all $\delta$ sufficiently small.
\end{proof}

\begin{proof}[\textbf{Proof of Proposition \ref{prop-mplevel}}]

Now,  using  Lemma \ref{Ffeq} and identity \eqref{tdeltaq+1}
\begin{multline}\label{idnova}
\max_{t\geq0} I_F(tV_{\delta,a}) =  I_F(t_\delta V_{\delta,a})=  \int_\Omega \overline{F}_\lambda(t_\delta \Delta V_{\delta,a})dx-\frac{t_\delta^{q+1}}{q+1} - \frac{\mu}{s+1} t_\delta^{s+1}|V_{\delta,a}|^{s+1}_{s+1}\\
 = \frac{p}{p+1}\int_\Omega f_\lambda^{-1}(t_\delta \Delta V_{\delta,a})t_\delta\Delta V_{\delta,a} dx-\lambda\frac{p-r}{p+1}\frac{\left|f_\lambda^{-1}(t_\delta \Delta V_{\delta,a})\right|^{r+1}_{r+1}}{r+1} -\frac{t_\delta^{q+1}}{q+1}- \frac{\mu t_\delta^{s+1}|V_{\delta,a}|^{s+1}_{s+1}}{s+1} \\
 = \frac{2}{N}\int_\Omega f_\lambda^{-1}(t_\delta \Delta V_{\delta,a})t_\delta\Delta V_{\delta,a} dx-\lambda\frac{p-r}{p+1}\frac{\left|f_\lambda^{-1}(t_\delta \Delta V_{\delta,a})\right|^{r+1}_{r+1}}{r+1}  - \frac{\mu(q-s)t_\delta^{s+1}|V_{\delta,a}|^{s+1}_{s+1}}{(q+1)(s+1)} .
\end{multline}

\noindent By \eqref{I'tv} and  Lemma \ref{estimateFtv} one gets
\begin{equation}\label{tdelta}
t_\delta^q\leq t_\delta^{1/p}S-Ct_\delta^{1/p}\delta^{\frac{N(p+1)}{q+1}}-t_\delta^s \mu|V_{\delta,a}|^{s+1}_{s+1}\quad \Rightarrow \quad t_\delta< S^\frac{p}{pq-1}.
\end{equation}
Combining this with \eqref{idnova} and  Lemma \ref{estimateFtv}, we infer that
\begin{equation}
\max_{t\geq0} I_F(tV_{\delta,a})< \displaystyle\frac{2}{N}t_\delta^\frac{p+1}{p}S- \frac{\mu(q-s)}{(q+1)(s+1)} t_\delta^{s+1}|V_{\delta,a}|^{s+1}_{s+1}< \frac{2}{N}S^\frac{pN}{2(p+1)},
\end{equation}
which concludes the proof.
\end{proof}

\begin{obs}\label{rmkfinal} For $N=3$, we mention that the estimates from \eqref{termoLaplacianoB} and \eqref{termoLaplacianoC} can be used to prove the existence of a positive solution to \eqref{sist} for the pairs $(p,q)$ on the critical hyperbola \eqref{pqHC} that are not included in Theorem \ref{theo1}, namely with $7/2< p<8$, and for some (not all) $(r,s)$ as in \eqref{rscondition}. This remark is linked to the condition $3 < t < 5$ in \cite[Corollary 2.3]{BrezisNirenberg1983} to prove the existence of a solution to \eqref{eq:BN}.
\end{obs}

By \eqref{I'tv} and  Lemma \ref{estimateFtv} one gets
\begin{equation}\label{tdelta}
t_\delta^q\leq t_\delta^{1/p}S+c_1t_{\delta}^\frac{r+1}{r} \delta^{\frac{3(r+1)}{r(p+1)}}+i_{\delta,a}-t_\delta^s \mu|V_{\delta,a}|^{s+1}_{s+1}.
\end{equation}
Combining this with \eqref{idnova} and  Lemma \ref{estimateFtv}, we infer that
\begin{equation}
\max_{t\geq0} I_F(tV_{\delta,a})\leq \displaystyle\frac{2}{3}t_\delta^\frac{p+1}{p}S+c_1t_\delta^\frac{r+1}{r} \delta^{\frac{3(r+1)}{r(p+1)}}+i_{\delta,a}- \frac{\mu(q-s)}{(q+1)(s+1)} t_\delta^{s+1}|V_{\delta,a}|^{s+1}_{s+1},
\end{equation}
and this is smaller than $\frac{2}{3}S^\frac{3p}{2(p+1)}$ if, and only if, 
\begin{equation}\label{eqcase2}
C\delta^{\frac{3(r+1)}{r(p+1)}}\leq c_1\mu|V_{\delta,a}|^{s+1}_{s+1}-i_{\delta,a} = c_1\mu|V_{\delta,a}|^{s+1}_{s+1} + |i_{\delta,a}|.
\end{equation}
To get \eqref{eqcase2} it is sufficient to verify, as $\delta\rightarrow0$, that
\begin{equation}\label{comp.r}
c \delta^\frac{3(r+1)}{r(p+1)}< |i_{\delta,a}|,
\end{equation}
or
 \begin{equation}\label{comp.s}
c \delta^\frac{3(r+1)}{r(p+1)}<c_1\mu|V_{\delta,a}|^{s+1}_{s+1}= \left\{\begin{array}{ll}
C\mu \delta^{\frac{3(s+1)}{p+1}},& \text { if } s<2, \\
C\mu \delta^{\frac{9}{p+1}}|\log \delta|,& \text { if } s=2, \\
C\mu \delta^{3-\frac{3(s+1)}{q+1}},& \text { if } s>2,
\end{array}\right. \end{equation}
where the behavior of $|V_{\delta,a}|_{s+1}$ in \eqref{comp.s} is given in \cite[eq. (36)]{EdJe-djairo}. 

To obatin \eqref{comp.r}, we keep all the calculation from Step 3.3 for the case with $7/2< p \leq (N+2)/(N-2)=5$. Then, we execute similar estimates for $5 < p < 8$. Therefore, for $7/2 < p <8$, using the residual terms in the first three lines (that is $q\geq\frac{N}{N-2}=3$) of \eqref{feitoh}, which is a refinement of \eqref{tudodei}, one can see that \eqref{comp.r} holds, iff
\begin{equation}\label{eq:obsN3}
 \begin{array}{rl}
r<\frac{1}{p-3} &\text{ if } r<2,\\
p\leq \frac{7}{2} &\text{ if } r=2,\\
0<r^2+(1-p)r+1&\text{ if } r>2,
\end{array}
\end{equation}
otherwise, the term of \eqref{feitoj} is dominant.

Let us now consider \eqref{comp.s}. For $s<2$ the inequality is equivalent to $\frac{r+1}{r}>s+1$, that is $rs<1$, which is a contradiction with condition \eqref{rscondition}. For $s=2$ the inequality is equivalent to $\frac{r+1}{r}\geq3$, that is $r\leq 1/2$, and this together with \eqref{rscondition}, gives $r=1/2$. Finally, for $s>2$ the inequality is true if 
\begin{equation}\label{comprs3}
\frac{3(r+1)}{r(p+1)}> 3-\frac{3(s+1)}{q+1} \quad\text{that is,} \quad s+1> q+1 - \frac{r+1}{r}\frac{q+1}{p+1}=\frac{q+1}{p+1}(p-\frac{1}{r}).
\end{equation}

Therefore, given any $(p,q)$ on the critical hyperbola \eqref{pqHC} with $N=3$, $7/2< p < 8$, $(r,s)$ as in \eqref{rscondition} with the one of the extra conditions \eqref{eq:obsN3}, $(r,s)= (1/2,2)$ or \eqref{comprs3}, the mountain pass level of $I_F$ is in the range of compactness and the mountain pass theorem ensures the existence of a solution.

\bibliographystyle{abbrv}

\end{document}